\documentclass[a4paper,12pt]{amsart}
\usepackage{latexsym,amscd,amssymb,url}
\usepackage[all,cmtip]{xy}  
\usepackage{graphicx}
\usepackage{xcolor}
\usepackage{enumerate} 
\definecolor{dblue}{rgb}{0,0,.6}
\usepackage[colorlinks=true, linkcolor=dblue, citecolor=dblue, filecolor = dblue, menucolor = dblue, urlcolor = dblue]{hyperref}

\newtheorem{theorem}{Theorem}[section]

\theoremstyle{plain}

\newtheorem{corollary}[theorem]{Corollary}

\newtheorem{example}[theorem]{Example}

\newtheorem{lemma}[theorem]{Lemma}

\newtheorem{proposition}[theorem]{Proposition}
\newtheorem{remark}[theorem]{Remark}

\newcommand{\del}{\partial}
\newcommand{\Z}{\mathbb Z}
\newcommand{\Q}{\mathbb Q}

\newcommand{\C}{\mathbb C}

\newcommand{\CP}{\mathbb P}

\newcommand{\Aut}{\operatorname{Aut}}

\newcommand{\id}{\operatorname{id}}

\newcommand{\Spec}{\operatorname{Spec}}

\newcommand{\Gal}{\operatorname{Gal}}

\newcommand{\Bir}{\operatorname{Bir}}

\newcommand{\codim}{\operatorname{codim}}

\newcommand{\CH}{\operatorname{CH}}
\newcommand{\supp}{\operatorname{supp}}
\newcommand{\sing}{\operatorname{sing}} 
\newcommand{\red}{\operatorname{red}}

\newcommand{\Frac}{\operatorname{Frac}}

\newcommand{\F}{\mathbb F}

 \newcommand{\slope}{\operatorname{slope}}

\newcommand{\dashedlongrightarrow}{\xymatrix@1@=15pt{\ar@{-->}[r]&}}
\renewcommand{\longrightarrow}{\xymatrix@1@=15pt{\ar[r]&}}
\renewcommand{\mapsto}{\xymatrix@1@=15pt{\ar@{|->}[r]&}}
\renewcommand{\twoheadrightarrow}{\xymatrix@1@=15pt{\ar@{->>}[r]&}}
\newcommand{\hooklongrightarrow}{\xymatrix@1@=15pt{\ar@{^(->}[r]&}}
\newcommand{\congpf}{\xymatrix@1@=15pt{\ar[r]^-\sim&}}
\renewcommand{\cong}{\simeq}

\begin{document}   
\title[Stably irrational hypersurfaces of small slopes]{Stably irrational hypersurfaces of small slopes}

\author{Stefan Schreieder} 
\address{Mathematisches Institut, LMU M\"unchen, Theresienstr.\ 39, 80333 M\"unchen, Germany.}
\email{schreieder@math.lmu.de}

\date{May 23, 2019; \copyright{\ Stefan Schreieder 2018}}
%


\begin{abstract}  
Let $k$ be an uncountable field of characteristic different from two.
We show that a very general hypersurface $X\subset \CP^{N+1}_k$ of dimension $N\geq 3$ and degree at least $ \log_2N +2$ is not stably rational over the algebraic closure of $k$. 
\end{abstract}

\maketitle

\section{Introduction}

A classical problem in algebraic geometry asks to determine which varieties are rational, i.e.\ birational to projective space. 
A very challenging and interesting class of varieties for this question 
are
smooth projective hypersurfaces of low degree. 
While the problem is solved in characteristic zero and dimension three by the work of Clemens--Griffiths \cite{clemens-griffiths} and Iskovskikh--Manin \cite{IM}, it is still wide open in higher dimensions. 

A measure for the complexity of the rationality problem for a smooth projective hypersurface $X$ is its slope:
$$
\slope(X):=\frac{\deg(X)}{\dim(X)+1}.
$$
If $\slope(X)>1$, then $H^0(X,\omega_X)\neq 0$ and so $X$ is not even separably uniruled.  

Generalizing the method of Iskovskikh--Manin to higher dimensions, Pukhlikov \cite{Pu1,Pu2} in low dimensions and de Fernex in general \cite{deF1,deF2} have shown that a smooth complex projective hypersurface $X$ of slope $1$ and 
dimension at least three 
is birationally rigid.
Again this is much stronger than proving irrationality as it implies for instance 
$\Aut(X)=\Bir(X)$. 

Using an entirely different method, which relies on the existence of regular differential forms on certain degenerations to positive characteristic, Koll\'ar \cite{kollar} showed that a very general complex projective hypersurface $X$ of degree at least $2\lceil \frac{ \dim(X)+3}{3} \rceil$ is not ruled, hence not rational.
Recently, Totaro \cite{totaro} combined this argument with the specialization method of Voisin and Colliot-Th\'el\`ene--Pirutka \cite{voisin,CT-Pirutka} to show that a very general complex projective hypersurface $X$ of degree at least $2\lceil \frac{\dim(X)+2}{3}\rceil$ is not stably rational, i.e.\ $X\times \CP^m$ is irrational for all $m$.
Totaro's result generalized \cite{CT-Pirutka}, where it was shown earlier that a very general complex quartic threefold is not stably rational.  
 
The method of Clemens--Griffiths has been generalized by Murre \cite{Mur} to threefolds over any field of characteristic different from $2$.
In particular, he has shown that over any such field, smooth cubic threefolds are irrational.
Similarly, the arguments of Colliot-Th\'el\`ene--Pirutka in \cite{CT-Pirutka} work over any uncountable field of characteristic different from two and so very general quartic threefolds are stably irrational over any such field.
In contrast, Koll\'ar and Totaro's method \cite{kollar,totaro} seems to work only 
over fields of small characteristic, compared to the dimension, and one gets the best bounds in characteristic zero and two.
Besides those results, not much seems to be known about the rationality problem for smooth hypersurfaces in positive characteristic.
For instance, to the best of my knowledge, up till now it was unknown whether for $N\geq 4$ there are smooth irrational Fano hypersurfaces in $\CP^{N+1}$ over algebraically closed fields of large characteristic, compared to $N$. 

\subsection{Main result}

Before this paper, no smooth projective hypersurface $X$ of slope at most $\frac{2}{3}$ was known to be irrational over algebraically closed fields. 
On the other hand, it is conjectured that at least over the complex numbers there should be smooth hypersurfaces of arbitrary small slopes (and in fact cubics) that are not stably rational. 
In this paper we produce irrational smooth hypersurfaces (e.g.\ over $\C$) whose degree grows logarithmically in the dimension, thus solving the above conjecture.

To state our result, note that the disjoint intervals $[2^{n-1} + n-2, 2^n + n-1)$ for positive integers $n\geq 2$ cover $[2, \infty)$, and so any integer $N \ge 3$ can be uniquely written as $n + r$ for integers $n \ge 2$ and $r \ge 1$ with $2^{n-1}-2 \le r  \le 2^n - 2$. 

\begin{theorem}\label{thm:main:intro}
Let $k$ be an uncountable field of characteristic different from two.
Let $N\geq 3$ be an integer and write $N=n+r$ with $2^{n-1}-2\leq r\leq 2^n-2$. 
Then a very general hypersurface $X\subset \CP^{N+1}_{k}$ of degree $d\geq n+2$ is not stably rational over the algebraic closure of $k$.
\end{theorem}

The following table illustrates our lower bounds in dimensions $N\leq 1032$. 

\medskip

\begin{center}
  \begin{tabular}{ |l | c | c | c | c |  c | c | c | c | r | }
    \hline
   $\dim(X)$ & $ \leq 4$ & $\leq 9$ & $\leq 18$ & $\leq 35$ & $\leq 68$ & $\leq 133$ & $\leq 262$ & $\leq 519$ & $\leq 1032$ \\ \hline
    $\deg(X)$ & $\geq 4$ & $\geq 5$ & $\geq  6$ & $\geq 7$ & $\geq 8$ & $\geq 9$ & $\geq 10$ & $\geq 11$ & $\geq 12$ \\
    \hline
  \end{tabular}
\end{center}
\medskip 
For $N=3$, we recover the result of Colliot-Th\'el\`ene--Pirutka \cite{CT-Pirutka} and for $N=4$, our bound coincides with that of Totaro \cite{totaro}.
However, in all dimensions at least $5$, our bounds are smaller than what was previously known.  
For instance, it was unknown whether complex quintic fivefolds are rational. 

If we write an integer $N\geq 3$ uniquely as $N=n+r$ with $2^{n-1}-2 \le r  \le 2^n - 2$  as in Theorem \ref{thm:main:intro}, then  $n\leq \lceil\log_2 N\rceil$.
Therefore, Theorem \ref{thm:main:intro} implies the following. 
%
 
\begin{corollary} \label{cor:thm:main}
Let $k$ be an uncountable field of characteristic different from two.
A very general hypersurface $X\subset \CP^{N+1}_k$ of dimension $N\geq 3$ and degree at least  $\log_2N +2$ is not stably rational over the algebraic closure of $k$.
\end{corollary} 

While \cite{kollar,totaro} produced a linear lower bound on the degree, our lower bound grows only logarithmically in $N$ and so we get surprisingly strong results in high dimensions. 
For instance, over any uncountable field of characteristic different from two, a very general hypersurface of dimension $N\leq 1\: 048\: 594$ and degree at least $22$ is not stably rational.

%
%

\subsection{Explicit equations} \label{subsec:equations}
It is possible to 
write down explicit equations for the examples in Theorem \ref{thm:main:intro} over countable fields $k$.
As our proof uses a new double degeneration argument, this works e.g.\ over  fields admitting two degenerations, such as $\Q(t)$ or $\F_p(s,t)$.
In Appendix \ref{app:Examples}, we give explicit examples in arbitrary dimension, and for all degrees covered by Theorem \ref{thm:main:intro}.
We illustrate this now for $k=\Q(t)$. 

For this, let $N\geq 3$ be an arbitrary integer.
As in Theorem \ref{thm:main:intro}, there are unique integers $n\geq 2$ and $r\geq 1$ with $N=n+r$ and $2^{n-1}-2\leq r\leq 2^n-2$.
Fix an integer $d\geq n+2$.
(Any integer $d\geq \log_2N+2$ has this property.)
For simplicity, we additionally assume that $d$ is even, but similar examples also exist for odd $d$, see Appendix \ref{app:Examples}.

For any $\epsilon \in \{0,1\}^n$, we define $|\epsilon|:=\sum_{i=1}^n \epsilon_i$ and $\rho(\epsilon):=1+\sum_{i=1}^{n} (1- \epsilon_i)\cdot 2^ {i-1}$.
The latter yields a bijection $\rho:\{0,1\}^n\to \{1,\dots ,2^n\}$ and we put
$
S:=\rho^{-1}(\{1,2,\dots, r+1\}).
$
Let $t\in \C$ be a transcendental number (e.g.\ $\pi$ or $e$), and let $p\neq q$ be odd primes with $q\nmid d$.
Then the hypersurface
$
X\subset \CP^{N+1}_{\Q(t)}
$ of dimension $N$ and even degree $d\geq n+2$, given by the homogeneous polynomial 
\begin{align*}
q\cdot \left(-x_0^{d-n}x_1x_2\cdots x_n
+t^2\left( \sum_{i=0}^n x_i^{d/2}\right)^2 +\sum_{\epsilon \in S}
 x_0^{d-2-|\epsilon|}\cdot x_1^{\epsilon_1}\cdots x_n^{\epsilon_n} \cdot  x^2_{n+\rho(\epsilon)}\right) 
 + p\cdot \sum_{i=0}^{N+1}x_i^d ,
\end{align*}
is smooth and not stably rational over $\C$.

If the dimension $N$ is of the special form $2^n+n-2$, then we can circumvent one of the degenerations in our argument, giving rise to examples over fields like $\Q$ and $\F_p(t)$.
For instance, if $N=2^n+n-2$, the examples over $\Q$ will be obtained from the above equation by setting $t=1$.
This leads to the following result.

\begin{theorem} \label{thm:slope}
Let $k$ be a field of characteristic different from two.
If $k$ has positive characteristic, assume that it has positive  transcendence degree over its prime field.
Then there are smooth projective hypersurfaces over $k$ of arbitrarily small slopes that are stably irrational over the algebraic closure of $k$. 
\end{theorem}

\subsection{Unirational hypersurfaces} 
Up till now, there was no example of a smooth projective unirational hypersurface over an algebraically closed field which was known to be stably irrational.
This is slightly surprising and reflects the difficulty of the (stable) rationality problem for smooth hypersurfaces, as for other types of varieties, many unirational but stably irrational examples are known, see e.g.\ \cite{artin-mumford,CTO,asok,voisin,HKT,HPT2,Sch1,Sch2}.

We prove in fact a strengthening of Theorem \ref{thm:main:intro}, where we allow the hypersurface to have some given multiplicity along a linear subspace, see Theorem \ref{thm:main} and Corollary \ref{cor:linearspace}.
Together with the unirationality result from \cite{CMM}, we then obtain the following.

\begin{corollary}\label{cor:unirational}
Let $N\in\{6,7,8,9\}$. 
Then a very general quintic hypersurface $X\subset \CP^{N+1}_{\C}$ containing a $3$-plane is a smooth  hypersurface that is unirational but not stably rational.
\end{corollary}

\subsection{The integral Hodge conjecture for rationally connected varieties}
In \cite{voisin-integral-hodge}, Voisin proved the integral Hodge conjecture (IHC) for uniruled threefolds, hence for rationally connected ones.
Later, Voisin asked whether the IHC for codimension two cycles holds for rationally connected varieties in arbitrary dimension and conjectured that the answer is negative in dimensions at least four, see \cite[Question 16]{voisin-takagi}.

Colliot-Th\'el\`ene and Voisin \cite{CTV} showed subsequently that the failure of the IHC for codimension two cycles on a rationally connected smooth complex projective variety $X$ is detected by the third unramified cohomology of $X$.
Using the six-dimensional example in \cite{CTO}, Colliot-Th\'el\`ene and Voisin then concluded that the IHC for rationally connected varieties of dimension at least six in general fails \cite{CTV}.
In the same article, they asked again about the case of rationally connected varieties of dimensions four and five \cite[Question 6.6]{CTV}.
For special types of rationally connected four- and fivefolds (including the case of cubics), a positive answer to that question is known to hold, see e.g.\ \cite[Theorem 18]{voisin-takagi}, \cite[Th\'eor\`eme 6.8]{CTV}, 
\cite[Theorem 1.4]{voisin-JAG} and \cite[Th\'eor\`eme 3]{FT}.

As a byproduct of our proof of Theorem \ref{thm:main:intro}, we obtain the following result, which partially answers a question of Asok \cite[Question 4.5]{asok}, and, by \cite{CTV}, completely answers the above mentioned question of Voisin and Colliot-Th\'el\`ene--Voisin.

\begin{theorem} \label{thm:unramified:coho}
For integers $N$ and $i$ with $2\leq i\leq N-1$, there is a unirational smooth complex projective variety $X$ of dimension $N$ with non-trivial $i$-th unramified cohomology: 
$$
H^ i_{nr}(\C(X)\slash \C,\Z/2)\neq 0.
$$ 
\end{theorem}


\begin{corollary} \label{cor:IHC}
In any dimension at least $4$, there is a smooth complex projective unirational variety for which the integral Hodge conjecture for codimension two cycles fails.
\end{corollary}
 
Note the the examples used in the above results are (weak) conic bundles and not hypersurfaces, see Section \ref{subsec:IHC:proof} below.
For instance, the four-dimensional example in Corollary \ref{cor:IHC} is a (weak) conic bundle over $\CP^3$.
  
  \subsection{Method}
Instead of degenerations to mildly singular varieties in characteristic two, used by Koll\'ar \cite{kollar} and Totaro \cite{totaro}, we use in this paper a degeneration to a highly singular hypersurface $Z\subset \CP^{N+1}$ (corresponding to $p\to 0$ in the equation in Section \ref{subsec:equations}). 
In fact, the singularities of $Z$ are so bad that the degeneration method of Voisin \cite{voisin} and Colliot-Th\'el\`ene--Pirutka \cite{CT-Pirutka} that has been used in \cite{totaro} does not seem to apply, see Remark \ref{rem:CH0} below.
Instead, Theorem \ref{thm:main:intro} is an application of the method that I have introduced in \cite{Sch1} and which generalizes \cite{voisin,CT-Pirutka} to degenerations where much more complicated singularities are allowed.

One important condition which the degeneration methods in \cite{voisin,CT-Pirutka} and \cite{Sch1} have in common is the existence of some specialization $Z$ of the varieties we are interested in, such that stable irrationality for $Z$ can be detected via some cohomological obstruction, e.g.\ via the existence of some nontrivial unramified cohomology class $\alpha\in H^n_{nr}(k(Z)/k,\Z/2)$, see \cite{CTO}. 
The key novelty of the strategy in \cite{Sch1} is however the observation that instead of a careful analysis of the singularities of $Z$, needed for the arguments in \cite{CT-Pirutka}, it suffices to check that the unramified class $\alpha$ restricts to zero on all exceptional divisors of a resolution of singularities of $Z$.
It is exactly this flexibility, that we will crucially exploit in this paper.

An additional difficulty arises in positive characteristic, where resolution of singularities is an open problem.
To be able to deal with such fields as well, we will develop in Section \ref{sec:degeneration} below an analogue of the method of \cite{Sch1} where one replaces a resolution of singularities of $Z$ by an alteration of suitable degree, which always exists by the work of de Jong and Gabber.
While the method from \cite{Sch1} can be adopted to alterations, it seems impossible to use alterations in the context of the original method of \cite{voisin,CT-Pirutka}. 

We will use a degeneration of a very general hypersurface of degree $d$ to a special hypersurface $Z\subset \CP^{N+1}$ of degree $d$ and multiplicity $d-2$ along an $r$-plane $P$.
Blowing up the $r$-plane, we get a (weak) $r$-fold quadric bundle $f:Y\to \CP^n$, cf.\ \cite[Section 3.5]{Sch1}, and we use that structure to produce a nontrivial unramified cohomology class $\alpha\in H^n_{nr}(k(Z)/k,\Z/2)$.
The first examples of quadric bundles with nontrivial unramified cohomology over $\CP^2$ (resp.\ $\CP^3$) and fibre dimension $r=1,2$ (resp.\ $r=3,\dots ,6$) have been constructed in \cite{artin-mumford,CTO}.
Recently, these results have been generalized to arbitrary $n,r\geq 1$ with $2^{n-1}-1\leq r\leq 2^n-2$ in \cite{Sch1}.

The main difficulties that we face are as follows.
Firstly, we need to find a nontrivial unramified cohomology class for a hypersurface $Z$ of small slope, while all previously known examples have large slopes, see \cite{Sch1}.
Secondly, the known methods from \cite{CTO,Sch1} do not seem to work in dimensions of the form $N=2^n+n-1$.
Finally, we have to arrange that $\alpha$ restricts to zero on all exceptional divisors of a resolution of $Z$, (or more generally on all subvarieties of an alteration of $Z$ that lie over the singular locus of $Z$). 
I have noticed before (cf.\ \cite{Sch1,Sch2}) that such a vanishing result is often automatically satisfied for all subvarieties that do not dominate $\CP^n$  and we prove a general such vanishing result in Theorem \ref{thm:vanishing} below. 
However, the key additional issue here is that $\alpha$ also has to restrict to zero on the (weak) $(r-1)$-fold quadric bundle $E\to \CP^n$ that we introduce in the blow-up $Y=Bl_PZ$ as exceptional divisor. 

In this paper we introduce a new construction method for quadric bundles with nontrivial unramified cohomology which circumvents all complications mentioned above at the same time.  
Our construction is inspired by 
 an example of a quadric surface bundle over $\CP^2$ that played a key role in the work of Hassett, Pirutka and Tschinkel \cite[Example 8]{HPT}. 
 An important step in the argument is a degeneration of the quadric bundle $Y=Bl_PZ$ to a bundle with a section, hence to a rational variety, which allows us to control the unramified cohomology of $Y$, see Section \ref{sec:non-vanishing} below.
Together with the initial degeneration to the singular hypersurface $Z$, this yields a double degeneration argument, which is the main technical innovation of the paper.

\section{Preliminaries}

\subsection{Conventions}
A variety is an integral separated scheme of finite type over a field.
For a scheme $X$, we denote its codimension one points by $X^{(1)}$.
A property holds for a very general point of a scheme if it holds at all closed points inside some countable intersection of open dense subsets.  
A quadric bundle is a flat projective morphism $f:Y\to S$ of varieties whose generic fibre is a smooth quadric; if we drop the flatness assumption, $Y$ is called a weak quadric bundle.

\subsection{Alterations} \label{subsec:alteration}
Let $Y$ be a variety over an algebraically closed field $k$.
An alteration of $Y$ is a proper generically finite surjective morphism $\tau:Y'\to Y$, where $Y'$ is a non-singular variety over $k$.
De Jong proved that alterations always exist, see \cite{deJo}. 
Later, Gabber showed that one can additionally require that $\deg(\tau)$ is prime to any given prime number $\ell$ which is invertible in $k$, see \cite[Theorem 2.1]{IT}.

\subsection{Galois cohomology and unramified cohomology}
Let $\ell$ be a prime and let $K$ be a field of characteristic different from $\ell$ which contains all $\ell$-th roots of unity.
We identify the Galois cohomology group $H^i(K,\Z/\ell)$ with the \'etale cohomology $H^i_{\text{\'et}}(\Spec K,\Z/\ell)$, where $\Z/\ell$ denotes the constant sheaf. 
We have $H^1(K,\Z/\ell)\cong K^\ast/(K^\ast)^\ell$ via Kummer theory.
Using this isomorphism, we denote by $(a_1,\dots ,a_i)\in H^i(K,\Z/\ell)$ the cup product of the classes $(a_j)\in H^1(K,\Z/\ell)$, represented by $a_j\in K^\ast$.
If $K$ has transcendence degree $d$ over an algebraically closed subfield $k\subset K$, then $H^i(K,\Z/\ell)=0$ for all $i>d$, see \cite[II.4.2]{serre}.

For any discrete valuation ring $A$ with residue field $\kappa$ and fraction field $K$,
 both of characteristic different from $\ell$, there is a residue map $$
 \del_A:H^i(K,\Z/\ell)\to H^{i-1}(\kappa,\Z/\ell) .
 $$ 
 This has the following property, see e.g.\ \cite[Lemma 9]{Sch1}.

\begin{lemma}\label{lem:residue}
In the above notation, suppose that $-1\in (K^\ast)^\ell$.
Let $\pi\in A$ be a uniformizer, $0\leq m\leq i$ be integers and let $a_1\dots ,a_{i}\in A^\ast$ be units in $A$. 
Then  
$$
\del_A(\pi a_1,\dots ,\pi a_m ,a_{m+1},\dots ,a_i)= \left( \sum_{j=1}^m (\overline a_1,\dots ,\widehat {\overline a_j},\dots ,\overline a_m) \right) \cup (\overline a_{m+1},\dots ,\overline a_i)  ,
$$
where $\overline a_j\in \kappa$ denotes the image of $a_j$ in $\kappa$, $(\overline a_1,\dots ,\widehat {\overline a_j},\dots ,\overline a_m)$ denotes the symbol where $\overline a_j$ is omitted, 
and where we use the convention that the above sum $\sum_{j=1}^m$ is one if $m=1$ and zero if $m=0$. 
\end{lemma}
\begin{proof}
The cases $m=0,1$ follow from \cite[Proposition 1.3]{CTO}.
For $m\geq 2$, the lemma follows from
$$
(\pi a_1,\dots ,\pi a_m ,a_{m+1},\dots ,a_i)=\left( \sum_{j=0}^m (a_1,\dots ,a_{j-1},\pi,a_{j+1},\dots ,a_m) \right) \cup (a_{m+1},\dots ,a_i) ,
$$ 
where the summand for $j=0$ is understood to be $(a_1,\dots ,a_m)$.
The latter identity follows from $(\pi,\pi)=0$, which itself is a consequence of the well--known relation  $(\pi,-\pi)=0$ (see e.g.\ \cite[Lemma 2.2]{kerz}) and the assumption $-1\in (K^\ast)^\ell$.
\end{proof}

Assume now that $K=k(X)$ is the function field of a normal variety $X$ over a field $k$. 
The unramified cohomology group $H^i_{nr}(K/k,\Z/\ell)$ is the subgroup of $H^i(K,\Z/\ell)$ that consists of all elements $\alpha\in H^i(K,\Z/\ell)$ that have trivial residue at any geometric discrete rank one valuation on $K$ that is trivial on $k$.\footnote{
We follow the convention used in \cite{merkurjev}, which slightly differs from \cite{CTO}, where also non-geometric valuations are considered.
Both definitions coincide by \cite[Theorem 4.1.1]{CT} if $X$ is smooth and proper.} 
If $x\in X$ is a scheme point in the smooth locus of $X$, then any $\alpha\in H^i(K,\Z/\ell)$ that is unramified over $k$ comes from a class in $H^i_{\text{\'et}}(\Spec \mathcal O_{X,x},\Z/\ell)$ and so it can be restricted to yield a class $\alpha|_{x}\in H^i(\kappa(x),\Z/\ell)$, see \cite[Theorem 4.1.1]{CT}.
That is, any $\alpha \in H^i_{nr}(K/k,\Z/\ell)$ can be restricted to the generic point of any subvariety $Z\subset X$ which meets the smooth locus of $X$.

\subsection{Quadratic forms}
Let $K$ be a field of characteristic different from $2$.
For $c_i\in K^\ast$, we denote by $\langle c_0, c_1,\dots ,c_{r+1}\rangle$ the quadratic form $q=\sum c_iz_i^2$ over $K$.
The orthogonal sum (resp.\ tensor product) of two quadratic forms $q$ and $q'$ over $K$ will be denoted 
by $q\perp q'$ (resp.\ $q\otimes q'$).
We say that $q$ and $q'$ are similar, if there is some $\lambda\in K^*$ with $q\cong \lambda q':=\langle\lambda \rangle \otimes q'$.
For any field extension $F$ of $K$ and any quadratic form $q$ over $K$ such that $\{q=0\}$ is integral over $F$, we denote by $F(q)$ the function field of the projective quadric over $F$ that is defined by $\{q=0\}$.

A quadratic form over $K$ is called Pfister form, if it is isomorphic to the tensor product of forms $\langle 1,-a_i\rangle$ for $i=1,\dots ,n$, where $a_i\in K^\ast$; if $-1$ is a square in $K$, then we may ignore the sign. 
As usual, we denote this tensor product by $\langle\langle a_1,\dots,a_n\rangle\rangle$.
Isotropic Pfister forms are hyperbolic, see e.g.\ \cite[Theorem X.1.7]{lam2}.

The following result is due to the work of many people, including Arason, Elman, Lam, Knebusch and Voevodsky.

\begin{theorem} \label{thm:psi:anisotrpoic}
Let $K$ be a field with $\operatorname{char}(K)\neq 2$ and let $a_1,\dots ,a_n\in K^\ast$.
The Pfister form $\psi=\langle\langle a_1,\dots ,a_n\rangle\rangle$ is isotropic if and only if $(a_1,\dots ,a_n)=0\in H^n(K,\Z/2)$.
\end{theorem}
\begin{proof}
The theorem follows from  \cite[Main Theorem 3.2]{EL} and Voevodsky's proof of the Milnor conjecture \cite{Voe}. 
\end{proof}

\begin{theorem} \label{thm:EL}
Let $K$ be a field with $\operatorname{char}(K)\neq 2$ and let $f:Q\to \Spec K$ be an integral projective quadric, defined by a quadratic form $q$ over $K$.
Let $a_1,\dots ,a_n\in K^\ast$ and consider $\alpha:=(a_1,\dots ,a_n)\in H^n(K,\Z/2)$.
Assume $\alpha\neq 0$.
Then the following are equivalent:  
\begin{enumerate}
\item $f^\ast \alpha=0 \in H^n(K(Q),\Z/2)$; \label{item:thm:EL:0}
\item the Pfister form $\psi:=\langle\langle a_1,\dots ,a_n\rangle\rangle$ becomes isotropic over $K(q)=K(Q)$; \label{item:thm:EL:1}
\item $q$ is similar to a subform of the Pfister form $\psi:=\langle\langle a_1,\dots ,a_n\rangle\rangle$. \label{item:thm:EL:2}
\end{enumerate}
\end{theorem}
\begin{proof} 
The equivalence of (\ref{item:thm:EL:0}) and (\ref{item:thm:EL:1}) follows from Theorem \ref{thm:psi:anisotrpoic}.
Since $\alpha\neq 0$, $\psi$ is anisotropic over $K$ by Theorem \ref{thm:psi:anisotrpoic}.
The equivalence of (\ref{item:thm:EL:1}) and (\ref{item:thm:EL:2}) is thus a consequence of the subform theorem of Arason and Knebusch, see 
\cite[Corollary X.4.9]{lam2}.
\end{proof}

\subsection{Decompositions of the diagonal}
We say that a variety $X$ admits an integral decomposition of the diagonal, if $\Delta_X=[z\times X]+B$ in $\CH_{\dim(X)}(X\times X)$ for some zero-cycle $z\in \CH_0(X)$ of degree one and some cycle $B$ with $\supp(B)\subset X\times S$ for some closed algebraic subset $S\subsetneq X$.
Equivalently, $\delta_X=[z\times {k(X)}]$ in $\CH_0(X\times {k(X)})$, where $\delta_X$ is the class of the diagonal and $z\times {k(X)}$ is the base change of the zero-cycle $z$ to the function field $k(X)$.
Sometimes, we will also write $X_{k(X)}:=X\times {k(X)}$ and $z_{k(X)}:=z\times {k(X)}$ for the corresponding base changes. 

Recall that a variety $X$ is called retract rational, if there are nonempty open subsets $U\subset X$ and $V\subset \CP^N$, for some integer $N$, and morphisms $f:U\to V$ and $g:V\to U$ with $g\circ f=\id_U$.
It is known (and not hard to see) that stably rational varieties are retract rational.
We have the following lemma, which in the case where $X$ is smooth and proper is due to Colliot-Th\'el\`ene and Pirutka \cite[Lemme 1.5]{CT-Pirutka}.

\begin{lemma} \label{lem:dec}
Let $X$ be a proper variety over a field $k$.
If $X$ is retract rational (e.g.\ stably rational), then it admits an integral decomposition of the diagonal. 
\end{lemma}

\begin{proof}
Suppose that there are nonempty open subsets $U\subset X$ and $V\subset \CP^N$, for some integer $N$, and morphisms $f:U\to V$ and $g:V\to U$ with $g\circ f=\id_U$.
Let $\Gamma_f\subset X\times \CP^N$ and $\Gamma_g\subset \CP^N\times X$ be the closures of the graphs of $f$ and $g$, respectively.
Let $K:=k(X)$ be the function field of $X$ and consider the diagram
$$
\xymatrix{
& \ar_{p}[ld] \ar^{q}[rd] \Gamma_f \times K & & \ar_{r}[ld] \ar^{s}[rd] \Gamma_g \times K & \\
X\times K& & \CP^N\times K & & X\times K ,}
$$
where $p$, $q$, $r$ and $s$ denote the natural projections, respectively.
Since $\Gamma_f$ and $X$ are birational, $K=k(\Gamma_f)$ and so the diagonal of $\Gamma_f$ gives rise to a zero-cycle $\delta_{\Gamma_f}$ on $\Gamma_f \times K$.
Since $q$ and $s$ are proper, the pushforwards $q_\ast$ and $s_\ast$ are defined on the level of Chow groups. 
There is also a refined Gysin homomorphism $r^{!}:\CH_0(\CP^N_ K)\to \CH_0((\Gamma_g)_ K)$, defined as follows, see  \cite[Definition 8.1.2]{fulton}.
Since $\CP^N$ is smooth, the graph $\Gamma_r\subset (\Gamma_g)_K\times \CP^N_K $ is a regularly embedded closed subvariety. 
For a cycle $z$ on $\CP^N_K$, the cycle $r^{!}(z)$ is then defined as intersection of $\Gamma_r$ with $(\Gamma_g)_K\times z$ (viewed as a cycle on $\Gamma_r\cong (\Gamma_g)_ K$). 

We claim that
\begin{align} \label{eq:delta_Gamma_f}
s_\ast \circ r^{!} \circ q_\ast (\delta_{\Gamma_f})=\delta_X \in \CH_0(X\times K),
\end{align}
where $\delta_X$ is the class of  the $K$-point of $X\times K$ induced by the diagonal of $X$.
To see this, note that the pushforward $q_\ast ( \delta_{\Gamma_f}) \in \CH_0(\CP^N\times K)$ is represented by the generic point of the graph of $f$ inside $\CP^N\times X$.
In particular, $q_\ast (\delta_{\Gamma_f})$ lies inside the open subset $V\times K$ over which $r$ is an isomorphism.
Hence, $r^{!} \circ q_\ast (\delta_{\Gamma_f})$ is represented by the $K$-point of $\Gamma_g\times K$ that corresponds to the generic point of the graph of the rational map $X\dashrightarrow \Gamma_g$ induced by $f$. 
Hence, $ s_\ast \circ r^{!} \circ q_\ast (\delta_{\Gamma_f})$ corresponds to the generic point of the graph of the rational map $g\circ f:X\dashrightarrow X$, which is the diagonal, because $g\circ f=\id_U$.
We have thus proven that (\ref{eq:delta_Gamma_f}) holds, as we want.
(Note that all closed points considered above have residue field $K$ and the morphisms $q$, $r$ and $s$ induce isomorphisms between those residue fields, so no multiplicities show up in the above computations.)

On the other hand, $\CH_0(\CP^N\times K)\cong [z\times K]\cdot \Z$ is generated by the class of the $K$-point $z\times K$ for any $k$-point $z\in \CP^N_k$, and we may choose $z\in V$.
Since $q_\ast (\delta_{\Gamma_f})$ has degree one, we conclude
$q_\ast (\delta_{\Gamma_f})=[z\times K]$.
Since $r$ is an isomorphism above $V$, the $k$-point $z\in V$ gives rise to a unique $k$-point $z'\in \Gamma_g$ such that
$
r^{!} (q_\ast (\delta_{\Gamma_f}))=[z'\times K]$.
If $z'':=g(z')$ denotes the image of $z'$ in $X$, then we conclude
$$
s_\ast(r^{!} (q_\ast (\delta_{\Gamma_f})))=[z''\times K].
$$
The lemma then follows by comparing this with (\ref{eq:delta_Gamma_f}) above.
\end{proof}
 
 \subsection{Specializations}
We say that a variety $X$ over a field $L$ specializes (or degenerates) to a variety $Y$ over a field $k$, with $k$ algebraically closed, if there is a discrete valuation ring $R$ with residue field $k$ and fraction field $F$ with an injection of fields $F\hookrightarrow L$, such that the following holds.
There is a flat proper morphism $\mathcal X\longrightarrow \Spec R$, such that $Y$ is isomorphic to the special fibre $Y\cong \mathcal X\times k$ and  $X\cong \mathcal X\times L$ is isomorphic to the base change of the generic fibre $\mathcal X\times F$.
With this definition, we have for instance the following.  
Let $f:\mathcal X\to B$  be a flat proper morphism  of varieties over an algebraically closed uncountable field whose fibres $X_b:=f^{-1}(b)$ are integral.
Then the fibre $X_t$ over a very general point $t\in B$ degenerates to the fibre $X_0$ for any closed point $0\in B$,  
cf.\ \cite[\S 2.2]{Sch1}.

\section{Degeneration method} \label{sec:degeneration}

In previous degeneration methods \cite{voisin,CT-Pirutka,Sch1}, it was crucial that the special fibre $Y$ admits a resolution of singularities.
This leads to difficulties in positive characteristic, where resolutions of singularities are not known to exist in general.
In this Section we show that the method in \cite{Sch1} still works, if we replace resolutions by alterations $\tau:Y'\to Y$ of suitable degree, which exist in arbitrary characteristic by the work of de Jong and Gabber, see Section \ref{subsec:alteration} above. 
Here we have no control on the birational geometry of $Y'$; for instance, $Y'$ might be of general type and of positive geometric genus, even though $Y$ is rationally connected.
In particular, we cannot expect that $Y'$ admits a decomposition of the diagonal and so the method of \cite{voisin,CT-Pirutka} does a priori not work in this context.

\begin{proposition}\label{prop:degeneration}
Let $X$ be a proper geometrically integral variety over a field $L$ which degenerates to a proper variety $Y$ over an algebraically closed field $k$.
Let $\ell$ be a prime different from $\operatorname{char}(k)$ and let $\tau:Y'\to Y$ be an alteration whose degree is prime to $\ell$. 
Suppose that for some $i\geq 1$ there is a nontrivial class $\alpha \in H^i_{nr}(k(Y)/k,\Z/\ell)$ such that 
$$
(\tau^\ast\alpha)|_{E}=0\in H^i(k(E),\Z/\ell)\ \ \text{for any subvariety $E\subset \tau^{-1}(Y^{\sing})$.}
$$ 
Then $X$ does not admit an integral decomposition of the diagonal.
In particular, $X$ is not retract rational and hence not stably rational. 
\end{proposition}

Before we turn to the proof of the above result, let us remark the following.

\begin{remark} \label{rem:vanishing}
In many important examples, the vanishing condition in Proposition \ref{prop:degeneration} turns out to be automatically satisfied, see e.g.\ \cite{Sch1,Sch2} and Proposition \ref{prop:unramified-coho} below.
A quite general result in this direction is proved in Theorem \ref{thm:vanishing} of this paper, which makes it easy to apply the above proposition in many cases. 
\end{remark}

\begin{proof}[Proof of Proposition \ref{prop:degeneration}]
Replacing $X$ by its base change to the algebraic closure of $L$, we may assume that $L$ is algebraically closed.
By Lemma \ref{lem:dec}, $X$ admits an integral decomposition of the diagonal if it is retract rational or stably rational.
For a contradiction, we thus assume that $X$ admits an integral decomposition of the diagonal. 
Via the specialization homomorphism on Chow groups \cite[Section 20.3]{fulton}, we then conclude that there is a decomposition of the diagonal of $Y$.
We let $K=k(Y)$ be the function field of $Y$ and conclude that 
\begin{align} \label{eq:deltaY}
\delta_Y=[z_K]\in \CH_0(Y_K) , 
\end{align}
where $\delta_Y$ denotes the class of the diagonal  
and $z_K$ is the base change of a zero-cycle $z\in \CH_0(Y)$ of degree one.

Let $U\subset Y$ be the smooth locus of $Y$ and let $U':=\tau^{-1}(U)$.
We have the following commutative diagram
$$
\xymatrix{
U'_K \ar[d]^{\tau|_{U'}} \ar@{^{(}->}[r]^{j'} &Y'_K\ar[d]^{\tau}\\
U_K \ar@{^{(}->}[r]^j          &Y_K .}
$$
Since $j$ is flat, $j^\ast$ is defined on the level of Chow groups.
Since $U'_K$ and $U_K$ are smooth, $\tau|_{U'}^\ast$ is defined as well, see \cite[\S 8]{fulton}.
Applying $\tau|_{U'}^\ast\circ j^\ast$ to (\ref{eq:deltaY}), we thus get:
\begin{align} \label{eq:tau*deltaY}
\tau|_{U'}^\ast(j^\ast\delta_Y)=\tau|_{U'}^\ast(j^\ast[z_K])\in \CH_0(U'_K) .
\end{align}

We have $j^\ast [z_K]=[z''_K]$, where $z''_K$ denotes the base change of a zero-cycle $z''\in \CH_0(U)$ (not necessarily of degree one).
Let $z'\in \CH_0(U')$ be the pullback of $z''$ via the morphism $U'\to U$.
It then follows that 
\begin{align} \label{eq:tau*deltaY:2}
\tau|_{U'}^\ast(j^\ast[z_K])=[z'_K] \in \CH_0(U'_K) ,
\end{align}
where $z'_K=z'\times K$ denotes the base change of $z'$ to $K$.

Let $\Gamma_\tau\subset Y'\times Y$ be the graph of $\tau$. 
Let $\delta'_Y\in \CH_0(Y'_K)$ be the zero-cycle given by the generic point of $\Gamma_\tau$. 
Since $U'_K\to U_K$ is \'etale above a neighbourhood of the diagonal point, we find that 
\begin{align} \label{eq:tau*deltaY:3}
\delta'_Y|_{U'_K}=\tau|_{U'}^\ast(j^\ast\delta_Y) \in \CH_0(U'_K).
\end{align}
Applying the localization exact sequence \cite[Proposition 1.8]{fulton} to the inclusion $j':U'_K\hookrightarrow Y'_K$, we then conclude from (\ref{eq:tau*deltaY}), (\ref{eq:tau*deltaY:2}) and (\ref{eq:tau*deltaY:3}) that
\begin{align} \label{eq:deltaY:2}
\delta'_Y=[z'_K] + [\tilde z] \in \CH_0(Y'_K) , 
\end{align}
where $\tilde z$ is a zero-cycle on $Y'_K$ whose support is contained in $Y'_K\setminus U'_K$.

Recall that there is a bilinear pairing
$$
\CH_0(Y'_K)\times H^i_{nr}(k(Y')/k,\Z/\ell)\longrightarrow H^i(K,\Z/\ell),\ \ ([z],\beta)\mapsto \langle [z],\beta \rangle .
$$
If $z$ is a closed point of $Y'_K$, it is defined as follows.
Pulling back $\beta$ via $Y'\times Y\to Y'$ and noting that $k(Y'\times Y)=K(Y'_K)$ we obtain a class $\beta_K\in H^i(K(Y'_K),\Z/\ell)$ that is unramified over $k$ and hence also over the larger field $K$.
We may thus consider the restriction ${\beta_K}|_z\in H^i(\kappa(z),\Z/\ell)$ of $\beta_K$ 
to the closed point $z\in Y'_K$.
The  class $\langle [z],\beta \rangle \in H^i(K,\Z/\ell)$ is then given by pushing down $ {\beta_K}|_z$ via the finite morphism $\Spec \kappa (z) \to \Spec K$. 
Since $Y'$ is smooth and proper over $k$, this pairing descends from the level of cycles to Chow groups, see \cite[\S 2.4]{merkurjev}.

Let us now consider the class $\tau^\ast \alpha \in H^i_{nr}(k(Y')/k,\Z/\ell)$.
We aim to pair this class with $\delta_Y'$.
To this end, recall that the graph $\Gamma_\tau$ is isomorphic to $Y'$ and so the generic point of $\Gamma_\tau$, which represents $\delta_Y'$, has residue field $k(Y')$ and $\Spec k(Y')\to \Spec K$ is induced by $\tau$.
By the above description of the pairing, this implies that 
$$
\langle \delta_Y', \tau^\ast \alpha \rangle =\tau_\ast \tau^\ast \alpha=\deg(\tau) \alpha\in H^i(K,\Z/\ell) .
$$
This class is nonzero, because $\deg(\tau)$ is prime to $\ell$ and $\alpha \neq 0$.
On the other hand, using the decomposition of $\delta_Y'$ in (\ref{eq:deltaY:2}), we claim that
$$
\langle \delta_Y', \tau^\ast \alpha \rangle=\langle z'_K+\tilde z,\tau^\ast \alpha \rangle=0 ,
$$
which contradicts the previous computation, as we want.
To prove our claim, note that  $\langle z'_K,\tau^\ast \alpha \rangle=0$ because $z'_K$ is the base change of a zero-cycle $z'$ on $Y'$ and so this pairing factors through the restriction of $\tau^\ast \alpha$ to $z'\in \CH_0(Y')$, which vanishes because $H^i_{nr}(k/k,\Z/\ell)=0$ since $i\geq 1$ and $k$ is algebraically closed.
To see that $\langle \tilde z,\tau^\ast \alpha \rangle=0$, note that $\tilde z_K$ is supported on the complement of $U'_K$ in $Y'_K$ and so it suffices to see that $\langle y',\tau^\ast \alpha \rangle=0 $ for any closed point $y'\in Y'_K\setminus U'_K$.
The image of a closed point $y'\in Y'_K\setminus U'_K$ via $Y'_K\to Y'$ is the function field of a subvariety $Z'\subset Y'$ that is contained in $Y'\setminus U'$; that is, $Z'$ is a subvariety of $Y'$ that maps to the singular locus of $Y$.
The pairing  $\langle y',\tau^\ast \alpha \rangle$ factors through the restriction of $\tau^\ast \alpha$ to the function field of $Z'\subset Y'$ and so we conclude $\langle y',\tau^\ast \alpha \rangle =0$ because $(\tau^\ast \alpha) |_{Z'}=0$ by assumptions, as $Z'\subset \tau^{-1}( Y^{\sing})$. 
This proves the above claim, which finishes the proof of the proposition. 
\end{proof}

\begin{remark}
In the above notation, we may by \cite[Theorem 2.1]{IT} assume that the irreducible components of $Y'\setminus \tau^{-1}(U)$ are smooth. 
The injectivity property (see e.g.\ \cite[Theorems 3.8.1]{CT}) then implies that $\tau^\ast \alpha$ restricts to zero on  any subvariety $E\subset Y'$ that maps to $Y^{\sing}$, if and only if it restricts to zero on all components of $Y'\setminus \tau^{-1}(U)$.  
\end{remark}

\section{A special quadratic form} \label{sec:construction} 
Let $k$ be a field of characteristic different from $2$.
Let $n\geq 1$ be an integer and consider the function field $K:=k(\CP^n)$.
Let $x_0,\dots ,x_n$ be homogeneous coordinates on $\CP^n$.
For $i=1,\dots ,n$, we then consider the following rational function on $\CP^n$:
\begin{align} \label{def:ai}
a_i:=\frac{-x_i}{x_0}\in K^\ast.
\end{align}
Let $g\in k[x_0,\dots ,x_n]$ be a nontrivial homogeneous polynomial and put
\begin{align} \label{def:b}
b:=\frac{g}{x_0^{\deg(g)}}\in K^\ast .
\end{align}

We will always assume that $g$ satisfies the following two conditions.
Firstly,  
\begin{align}\label{eq:cond:codim}
\text{$g$ contains the monomial $x_i^{\deg(g)}$ nontrivially for all $i=0,\dots ,n$}.
\end{align} 
This condition is equivalent to asking that $g$ does not vanish at points of the form $[0:\dots:0:1:0:\dots :0]$, and hence not on any non-empty intersection of coordinate hyperplanes $\{x_{i_1}=x_{i_2}=\dots=x_{i_{c}}=0\} $.
Secondly,  we will assume that
\begin{align}\label{eq:cond:square}
\text{the image of $g$ in $k[x_0,\dots ,x_n]/(x_i)$ becomes a square for all $i=0,1,\dots ,n$.}
\end{align}

For $\epsilon=(\epsilon_1,\dots,\epsilon_n)\in \{0,1\}^n$, consider 
$$
c_\epsilon:=\prod_{i=1}^n x_i^{\epsilon_i} .
$$
Let further 
\begin{align} \label{eq:rho}
\rho:\{0,1,\dots , 2^n-1\}\stackrel{\sim}\longrightarrow \{0,1\}^{n}
\end{align}
be a bijection with $\rho(0)=(0,0,\dots ,0)$.
We put $c_i:=c_{\rho(i)}$ and $d_i:=\deg (c_i)$, and get $c_0=1$.

For $ r\leq 2^n-2$, we then define
\begin{align}
q:=\left\langle b,\frac{c_1}{x_0^{d_1}},\frac{c_2}{x_0^{d_2}},\dots ,\frac{c_{r+1}}{x_0^{d_{r+1}}} \right\rangle . \label{def:q} 
\end{align}
This quadratic form will play a key role in our arguments; it should be compared to the Pfister form
\begin{align} \label{def:psi}
\psi :=\langle\langle a_1,\dots , a_n \rangle\rangle =\left\langle 1,  \frac{c_1}{x_0^{d_1}},\frac{c_2}{x_0^{d_2}},\dots ,\frac{c_{2^n-1}}{x_0^{d_{2^n-1}}} \right\rangle . 
\end{align} 
By Theorem \ref{thm:psi:anisotrpoic}, the Pfister form $\psi$ is related to the class
\begin{align} \label{eq:alpha}
\alpha:=(a_1,\dots, a_n)\in H^n(K,\Z/2).
\end{align}

\begin{lemma}\label{lem:betaneq0}
We have $\alpha\neq 0\in H^n(K,\Z/2)$. 
\end{lemma}
\begin{proof} 
We use Lemma \ref{lem:residue} and take successive residues of $\alpha$ along $x_{n}=0$, $x_{n-1}=0$, $\dots$, $x_1=0$ to reduce the statement to the observation that $1\in H^0(k,\Z/2)$ is nonzero.
This proves $\alpha\neq 0$, as we want. 
\end{proof} 

\begin{example}  \label{ex:HPT}
If $n=2$, we may consider $g=x_0^2+x_1^2+x_2^2-2(x_0x_1+x_0x_2+x_1x_2)$, which defines a smooth conic $\{g=0\}\subset \CP^2$ that is tangent to the lines
$\{ x_i=0 \}$ for $i=0,1,2$.
In this case, conditions (\ref{eq:cond:codim}) and (\ref{eq:cond:square}) are satisfied.
For $r=2$, the corresponding quadratic form $q$ from (\ref{def:q}) coincides with the example of Hassett--Pirutka--Tschinkel in \cite[Example 8]{HPT}.
If $f:Q\to \Spec K$ denotes the corresponding projective quadric surface, then $f^\ast \alpha$ is nontrivial and unramified over $k$ by \cite[Proposition 11]{HPT}.
\end{example}

In the next section, we show that the projective quadric $f:Q\to \Spec K$ defined by $q$ in (\ref{def:q}) has always the property that $f^\ast\alpha$ is unramified over $k$, as long as (\ref{eq:cond:codim}) and (\ref{eq:cond:square}) hold.
We also prove that the vanishing condition needed for the degeneration method in Proposition \ref{prop:degeneration} is satisfied under these conditions.
Note however that conditions (\ref{eq:cond:codim}) and (\ref{eq:cond:square}) do not imply that $f^\ast \alpha$ is nontrivial.
In fact, since $\alpha\neq 0$, Theorem \ref{thm:EL} implies that $f^*\alpha$ is trivial if and only if $q$ is similar to a subform of $\psi$ (and this holds, for instance, when $b$ is a square).
If $r=2^n-2$, this last property is easily analysed:

\begin{lemma} \label{lem:bnotsquare}
Let $f:Q\to \Spec K$ be the projective quadric, defined by $q$ in (\ref{def:q}).
If $r=2^n-2$, then $f^\ast \alpha\neq 0$ if and only if $b$ is not a square in $K$.
\end{lemma} 
\begin{proof}
Since $r=2^n-2$, $q$ and $\psi$ have the same dimension.
Since $q$ and $\psi$ represent a common element, $q$ is similar to a subform of $\psi$ if and only if $q\cong \psi$ (see \cite[Theorem X.1.8]{lam2}) and this is by Witt's cancellation theorem equivalent to $b$ being a square in $K$. 
The lemma follows thus from Theorem \ref{thm:EL}, because $\alpha \neq 0$ by Lemma \ref{lem:betaneq0}.
\end{proof}

For $r< 2^n-2$, the question whether $q$ is similar to a subform of $\psi$ is quite subtle and so it is in general hard to decide whether $f^\ast  \alpha$ is nontrivial. 
For special choices of $g$, this problem will be settled later in Section \ref{sec:non-vanishing} below.

\section{A vanishing result}
\label{sec:vanishing}

\begin{proposition} \label{prop:unramified-coho}
Let $k$ be an algebraically closed field of characteristic different from $2$.
Let $n,r\geq 1$ be positive integers with $r\leq 2^n-2$. 
Let $f:Y\to \CP^n$ be a surjective morphism of proper varieties over $k$ 
whose generic fibre is birational to the quadric 
over $K=k(\CP^n)$ given by $q$ in (\ref{def:q}).
Assume that (\ref{eq:cond:codim}) and (\ref{eq:cond:square}) hold.
Then, 
\begin{enumerate}
\item $
f^\ast\alpha\in H^n_{nr}(k(Y)\slash k,\Z\slash 2),
$
where $\alpha\in H^n(K,\Z\slash 2)$ is from (\ref{eq:alpha}); \label{item:prop:unramified-coho:1}
\item 
for any dominant generically finite morphism $\tau:Y'\to Y$ of varieties and for any subvariety $E\subset Y'$ which meets the smooth locus of $Y'$ and which does not dominate $\CP^n$ via $f\circ \tau$, we have  \label{item:prop:unramified-coho:2}
$$
(\tau^\ast f^\ast\alpha)|_E=0\in H^n(k(E),\Z/2).
$$ 
\end{enumerate}
\end{proposition}

We will prove in Theorem \ref{thm:vanishing} below that (in a much more general setting) item (\ref{item:prop:unramified-coho:1}) in Proposition \ref{prop:unramified-coho}, i.e.\ the fact that $f^\ast\alpha $ is unramified, implies the vanishing in item (\ref{item:prop:unramified-coho:2}).
For sake of simplicity, we prefer not to invoke this general result in the following, but rely on a direct argument which uses the explicit description of the quadratic form $q$.

\begin{proof}[Proof of Proposition \ref{prop:unramified-coho}]  
Recall first that if (\ref{item:prop:unramified-coho:1}) holds, then $\tau^\ast f^\ast \alpha\in H^n_{nr}(k(Y')/k,\Z/2)$ (by functoriality of unramified cohomology)
and so the restriction $(\tau^\ast f^\ast\alpha)|_E$ in item (\ref{item:prop:unramified-coho:2}) is defined by \cite[Theorem 4.1.1(b)]{CT}. 
Assuming (\ref{item:prop:unramified-coho:1}), we claim that it suffices to prove (\ref{item:prop:unramified-coho:2}) in the case where $E$ is a divisor.
To see this, let $E\subset Y'$ be a subvariety which does not dominate $\CP^n$ and which meets the smooth locus of $Y'$.
By our conventions, $E$ is integral and so it is smooth at the generic point (because $k$ is algebraically closed).
This implies that the exceptional divisor of the blow-up $Bl_EY'$ has a unique component which dominates $E$ and this component is birational to $E\times \CP^{s}$, where $s=\dim(Y)-\dim(E)-1$.
Moreover, $Bl_EY$ is smooth at the generic point of this particular component.
Since $H^n(k(E),\Z/2)\to H^n(k(E\times \CP^s),\Z/2)$ is injective, replacing $Y'$ by $Bl_EY'$ thus shows that it suffices to prove (\ref{item:prop:unramified-coho:2}) in the case where $E\subset Y'$ is a divisor that does not dominate $\CP^n$.

By \cite[Proposition 1.7]{merkurjev}, we may up to birational modifications assume that $\tau(E)$ is a divisor on $Y$.
In order to prove that $(\tau^\ast f^\ast\alpha)|_E$ vanishes, it thus suffices to show that $f^\ast \alpha$ restricts to zero on the generic point of any prime divisor $E\subset Y$ with $f(E)\subsetneq \CP^n$.

Next, we claim that in order to prove item (\ref{item:prop:unramified-coho:1}), it suffices to show that $f^\ast\alpha$ has trivial residue at the generic point of any prime divisor $E\subset Y$ that does not dominate $\CP^n$.
To see this, let $\nu$ be a geometric discrete rank one valuation on $K(Q)$ that is trivial on $k$.
By \cite[Proposition 1.7]{merkurjev}, there is a normal variety $\widetilde Y$ and a dominant  birational morphism $\widetilde Y\to Y$ such that $\nu$ corresponds to a prime divisor on $\widetilde Y$.
Replacing $Y$ by $\widetilde Y$, we may thus assume that $\nu$ corresponds to a prime divisor $E$ on $Y$.
We denote its generic point by $y\in Y^{(1)}$.
If $E$ dominates $\CP^n$, then the residue at $y$ vanishes by Lemma \ref{lem:residue}: $\del_y(f^\ast \alpha)=0$.  
It thus suffices to treat the case where $f(E)\subsetneq \CP^n$, as claimed.

As we have seen above, in order to prove the proposition, it suffices to show 
\begin{align} \label{eq:delyf*alpha=0}
\del_y(f^\ast \alpha)=0,
\end{align}
and
\begin{align} \label{eq:f*alpha|E=0}
(f^\ast \alpha)|_E=0 \in H^n(k(E),\Z/2) ,
\end{align}
where $y\in Y^{(1)}$ denotes the generic point of a prime divisor $E\subset Y$ with $f(E)\subsetneq \CP^n$.

We will prove (\ref{eq:delyf*alpha=0}) and (\ref{eq:f*alpha|E=0}) simultaneously.
To begin with, we choose a normal projective variety $S$ with a birational morphism $S\to \CP^n$ such that $y\in  Y$ maps to a codimension one point $x\in S^{(1)}$ on $S$, 
cf.\ \cite[Propositions 1.4 and 1.7]{merkurjev} and \cite[Lemma 29]{Sch1}.

Let $c\geq 0$ be the maximal natural number such that $f(E)$ lies on the intersection of $c$ coordinate hyperplanes, that is, such that there are integers $0\leq i_1<i_2<\cdots < i_c \leq n$ with $x_{i_j}(f(y))=0$ for all $j=1,\dots ,c$. 
The proof proceeds now via two cases.

\textbf{Case 1.} The image $f(E)\subset \CP^n$ has dimension $\dim(f(E))=n-c$.

We first show that (\ref{eq:delyf*alpha=0}) and (\ref{eq:f*alpha|E=0}) follow from a different statement, that will be easier to check in this case.
To this end, consider the local rings $B:=\mathcal O_{Y,y}$ and $A:=\mathcal O_{S,x}$.
Let further $\hat A$ and $\hat B$ be the completions of $A$ and $B$, respectively, and let $\hat K=\Frac(\hat A)$ and $\hat L=\Frac(\hat B)$ be the corresponding fraction fields.
Since the generic fibre of $f:Y\to \CP^n$ is birational to the quadric defined by $q$ from (\ref{def:q}), inclusion of fields induces 
a sequence
\begin{align} \label{seq:K,Khat,Lhat}
H^n(K,\Z/2)\stackrel{\varphi_1}\longrightarrow H^n(\hat K,\Z/2)\stackrel{\varphi_2}\longrightarrow H^n(\hat K(q),\Z/2)\stackrel{\varphi_3}\longrightarrow H^n(\hat L,\Z/2) ,
\end{align} 
where we use that $\{q=0\}$ is integral over $\widehat K$ because $r\geq 1$.
The residue of $f^\ast \alpha$ at $y$ factors through the image of $\alpha$ in $H^n(\hat L,\Z/2)$ via the above sequence, see e.g.\ \cite[p.\ 143]{CTO}.
Moreover, if $\del_y(f^\ast \alpha)=0$, then $f^\ast \alpha \in H^n_{\text{\'et}}(\Spec \hat B,\Z/2)\subset H^n(\hat L,\Z/2) $ (see e.g.\ \cite[\S 3.3 and \S 3.8 ]{CT}) and so the restriction of $f^\ast \alpha$ to $E$ factors through the image of $\alpha$ in $H^n(\hat L,\Z/2)$ via the above sequence as well.
Hence, in order to prove (\ref{eq:delyf*alpha=0}) and (\ref{eq:f*alpha|E=0}), it suffices to establish  
\begin{align} \label{eq:phi1phi2phi3}
\varphi_3\circ \varphi_2\circ \varphi_1(\alpha)=0 \in H^n(\hat L,\Z/2) .
\end{align} 

Since we are in case 1, $\codim_{\CP^n}(f(E))=c$.
This implies $c\geq 1$ because $f(E)\subsetneq \CP^n$.
By the definition of $c$, $f(y)$ lies on the intersection of $c$ coordinate hyperplanes. 
The assumption $\codim_{\CP^n}(f(E))=c$ thus implies that $f(y)$ is the generic point of $\{x_{i_{1}}=\dots =x_{i_c}=0\}$ for some $0\leq i_1<i_2<\cdots < i_c \leq n$ and so 
 condition (\ref{eq:cond:codim}) implies 
\begin{align} \label{eq:gtildefynot0}
g(f(y))\neq 0 .
\end{align} 

There is some $j\in \{0,\dots ,n\}$ such that $x_j(f(y))\neq 0$.
Moreover, condition (\ref{eq:cond:square}) implies that $\deg(g)$ is even and so $b$ coincides with $b':=\frac{g}{x_j^{\deg(g)}}$ up to squares.
Since $c\geq 1$, 
condition (\ref{eq:cond:square}) implies that $\overline {b'}$ is a square in $\kappa(x)$.
By (\ref{eq:gtildefynot0}), it is in fact a nontrivial square and so Hensel's lemma implies that $b'$ (and hence also $b$) becomes a square in the field extension $\hat K$ of $K$, considered above.
Hence, over the field $\hat K$, $q$ becomes isomorphic to a subform of $\psi = \left\langle \left\langle a_1,\dots ,a_n\right\rangle \right\rangle $. 
By Theorem \ref{thm:EL}, we thus get
$$
\varphi_2(\varphi_1(\alpha))=0 .
$$  
Therefore, (\ref{eq:phi1phi2phi3}) holds and this implies (as we have seen above) (\ref{eq:delyf*alpha=0}) and (\ref{eq:f*alpha|E=0}).

\textbf{Case 2.} The image $f(E)\subset \CP^n$ has dimension $\dim(f(E))<n-c$.

In this case, consider the birational morphism $S\to \CP^n$ and think about $\alpha$ as a class on the generic point of $S$.
We aim to show 
\begin{align} \label{eq:del_xalpha}
\del_x\alpha=0 .
\end{align}
This will be enough to conclude (\ref{eq:delyf*alpha=0}) and (\ref{eq:f*alpha|E=0}), for the following reasons.
If  $\del_x\alpha=0$, then $f^\ast \alpha$ has trivial residue at $y$ (see e.g.\ \cite[p.\ 143]{CTO}) and so (\ref{eq:delyf*alpha=0}) holds.
Moreover, since $\del_x\alpha=0$, $(f^\ast\alpha)|_E$ can be computed by first restricting $\alpha$ to $\kappa(x)$ and then pulling it back to $k(E)$.
This implies $(f^\ast\alpha)|_E=0$ because  $H^n(\kappa(x),\Z/2)=0$ , since  $x\in S$ is a point of dimension $n-1$ over the algebraically closed ground field $k$. 

It thus remains to prove (\ref{eq:del_xalpha}).
To this end, we choose some $j\in \{0,\dots ,n \}$ such that $x_j(f(y))\neq 0$.
Multiplying each $a_i$ by the square of $x_0/x_j$, we get
$$
\alpha= \left( \frac{-x_0x_1}{x_j^2},\frac{-x_0x_2}{x_j^2},\dots, \frac{-x_0}{x_j} , \dots ,\frac{-x_0x_n}{x_j^2}\right)  .
$$
It is well-known that $(a,-a)=0$ for all $a\in K^\ast$ (see e.g.\ \cite[Lemma 2.2]{kerz}).
Applying this to $a=x_0/x_j$, the above identity yields:
$$
\alpha= \left( \frac{-x_1}{x_j},\frac{-x_2}{x_j},\dots, \frac{-x_0}{x_j},\dots  ,\frac{-x_n}{x_j} \right) .
$$
Hence, up to relabelling, we may assume that $j=0$ and so $x_0$ does not vanish at $f(y)$.
Up to relabelling further, we may also assume that $x_i(f(y))=0$ for $i=1,\dots ,c$ and $x_i(f(y))\neq 0$ for $i\geq c+1$.

If $c=0$, then (\ref{eq:del_xalpha}) is clear by Lemma \ref{lem:residue}.
If $c\geq 1$, Lemma \ref{lem:residue} implies that
$$
\del_x(\alpha)=\gamma_1\cup \gamma_2 ,
$$
with  
$
\gamma_2=(\overline {a_{c+1}},\dots ,\overline  {a_n}) \in H^{n-c}(\kappa(x),\Z/2) 
$,
where $\overline {a_i}$ for $i> c$ denotes the restriction of $a_i=\frac{x_i}{x_0}$ 
to $x$ (this works because $x_0$ and $x_i$ for $i>c$ do not vanish at $x$).
In particular, $\gamma_2$ is a pullback of a class from $H^{n-c}(\kappa( f(y)),\Z/2) $.
Hence, $\gamma_2=0$ because $k$ is algebraically closed and so the cohomological dimension of $\kappa( f(y))$ is less than $n-c$, as we are in case 2. 
This proves that (\ref{eq:del_xalpha}) holds, which finishes the proof in case 2. 

This concludes the proof of Proposition \ref{prop:unramified-coho}. 
\end{proof}

\section{A non-vanishing result} \label{sec:non-vanishing}

In this section we aim to construct examples of homogeneous polynomials $g$ that satisfy the conditions (\ref{eq:cond:codim}) and (\ref{eq:cond:square}) from Section \ref{sec:construction} in such a way that the unramified class $f^\ast \alpha$ from Proposition \ref{prop:unramified-coho} is nontrivial.

Let $k$ be a field of positive transcendence degree over its prime field $F\subset k$.
That is, there is some element $t\in k$ that is algebraically independent over $F$ and so $F(t)\subset k$. 

Let $n\geq 1$ be an integer and let $G\in F[x_0,\dots ,x_n]$ be a homogeneous polynomial of degree $\lceil\frac{ n+1}{2}\rceil$, which contains the monomial $x_i^{\lceil\frac{ n+1}{2}\rceil}$ nontrivially for all $i=0,\dots ,n$ (e.g.\ $G=\sum_{i=0}^n x_i^{\lceil\frac{ n+1}{2}\rceil}$).
We then consider
\begin{align} \label{eq:g:t}
g:=t^2G^2-x_0^{\epsilon}\cdot \prod_{i=0}^nx_i \ \ \text{and}\ \ b:=\frac{g}{x_0^{2 \lceil\frac{ n+1}{2}\rceil}} ,
\end{align}
where $\epsilon=0$ if $n+1$ is even and $\epsilon=1$ otherwise.
Since $g$ is a square modulo $x_i$ for all $0\leq i\leq n$, (\ref{eq:cond:square}) holds.
Since $G$ contains $x_i^{\lceil\frac{ n+1}{2}\rceil}$ nontrivially for all $i=0,\dots ,n$, condition (\ref{eq:cond:codim}) holds as well.

Let $f:Q\to \Spec K$ be the projective quadric defined by $q$ from (\ref{def:q}), where $b$ is as in (\ref{eq:g:t}) above.
Specializing $t\to 0$ shows that for any choice of $G$, the rational function $b$ is not a square in $k(\CP^n)$.
Hence, $f^\ast \alpha\neq 0$ if $r=2^n-2$, by Lemma \ref{lem:bnotsquare}.
If $r<2^n-2$, then this statement is in general not true any longer.
Indeed, 
$\langle b\rangle\subset \langle 1,\frac{-x_1\cdots x_n}{x_0^n}\rangle$ and so $q$ is a subform of $\psi$ if $-1$ is a square in $k$ and the monomial $x_1\cdots x_n$ is among the $c_i$ with $i=r+2,\dots ,2^n-1$ and the latter implies $f^\ast\alpha=0$ by Theorem \ref{thm:EL}.  
For what follows, it is therefore essential to assume 
that the bijection $\rho$ from (\ref{eq:rho}) is chosen in such a way that $\rho(1)=(1,1,\dots ,1)$ and so the following holds:  
\begin{align}
c_1=x_1\cdots x_n. \label{eq:c1:Constr3}  
\end{align}

\begin{proposition}\label{prop:cohononvanishing:Constr3}
Let $n,r\geq 1$ 
be integers with $r\leq 2^n-2$.
Let $k$ be a field of characteristic different from $2$ and of positive transcendence degree over its prime field $F$.
Let $K=k(\CP^n)$.
Let $f:Q\to \Spec K$ be the projective quadric defined by the quadratic form $q$ from (\ref{def:q}), where $g$ is as in (\ref{eq:g:t}).
Assume that (\ref{eq:c1:Constr3}) holds.
Then, $f^\ast \alpha \neq 0\in H^n(K(Q),\Z/2)$, where $\alpha =(a_1,\dots ,a_n)$. 
\end{proposition}

\begin{proof}
The idea is to specialize $t\to 0$.
Under this specialization, $b$ specializes to $-c_1/x_0^{n}$ and so $q$ becomes isotropic.
The specialization of the class $f^\ast \alpha$ is thus nonzero, as it is given by the pullback of a nonzero class (see Lemma \ref{lem:betaneq0}) via a purely transcendental field extension.
But then $f^\ast \alpha$ must be nonzero itself.
We give the details of this argument in what follows.

For a contradiction, we assume $f^\ast \alpha=0$. 
It follows for instance from Theorem \ref{thm:EL} that $\alpha$ maps to zero in $H^n(F(t)'(\CP^n)(q),\Z/2)$ for some finitely generated field extension $F(t)'$ of $F(t)$ with $F(t)'\subset k$.
We choose a normal $F$-variety $W$ which admits a surjective morphism $W\to \Spec F[t]$ such that $F(W)=F(t)'$ and $F(t)\hookrightarrow F(W)$ corresponds to the natural inclusion.
Since $W\to \Spec F[t]$ is surjective and $W$ is normal, we may after shrinking $W$ assume that there is a Cartier prime divisor $W_0\subset W$ which maps to the origin in $\mathbb A^1_F=\Spec F[t]$.


For $0\leq i\leq r+1$, we put $c'_i:=x_0^{2\lceil \frac{n+1}{2}\rceil-d_i}c_i$, where $d_i=\deg(c_i)$.
The quadratic form
$$
\varphi:=\langle g,c'_1,c'_2,\dots ,c'_{r+1} \rangle
$$
defines a subscheme
$$
\mathcal Q:=\{\varphi=0\}\subset  \CP^{r+1}_{W} \times _{W} \mathbb P^n_{W} 
$$
over $W$. 
The fibre $\mathcal Q_{F(W_0)}$ of  $\mathcal Q\to W$ above the generic point of $W_0\subset W$ is isomorphic to the quadric over $F(W_0)$ that is defined by the reduction $q_0$ of $q$ modulo $t$.
Since $q_0$ has full rank, $\mathcal Q_{F(W_0)}$ is smooth over $F(W_0)$.
Let $\eta_0\in \mathcal Q$ be the generic point of $\mathcal Q_{F(W_0)}\subset \mathcal Q$. 
Since the fibre $\mathcal Q_{F(W_0)}$ is reduced, the closure $\overline{\{\eta_0\}}\subset \mathcal Q$ of $\eta_0$ is an irreducible component of the pullback of the Cartier divisor $W_0\subset W$ via the natural map $\mathcal Q\to W$.
Hence, $\overline{\{\eta_0\}}$ is generically Cartier and so 
 $\eta_0$ is a smooth codimension one point of $\mathcal Q$.  
The local ring $B:=\mathcal O_{\mathcal Q, \eta_0}$ is therefore a discrete valuation ring with fraction field $F(W)(\CP^n)(q)$ and residue field $F(W_0)(\CP^n)(q_0)$. 

Since $a_i\in B^\ast$ for all $i$, $\alpha$ gives rise to a class in $H^n_{\text{\'et}}(\Spec B,\Z/2)$.
This class vanishes by assumptions, because $F(W)=F(t)'$ and restriction to the generic point yields an injection 
$
H^n_{\text{\'et}}(\Spec B,\Z/2)\hookrightarrow H^n(F(W)(\CP^n)(q),\Z/2) ,
$
see e.g.\ \cite[\S 3.6]{CT}.
Restricting that class to the closed point of $\Spec B$ then shows that $\alpha$ maps to zero in $H^n(F(W_0)(\CP^n)(q_0),\Z/2)$.
By (\ref{eq:g:t}) and (\ref{eq:c1:Constr3}), $g$ and $-c_1'$  coincide modulo $t$ and so $q_0$ is isotropic over $F(W_0)(\CP^n)$.
Hence, $H^n(F(W_0)(\CP^n),\Z/2)\to H^n(F(W_0)(\CP^n)(q_0),\Z/2)$ is injective, see e.g.\ \cite[Proposition 4.1.4]{CT}, and so $\alpha$ vanishes in $H^n(F(W_0)(\CP^n),\Z/2)$, which contradicts Lemma \ref{lem:betaneq0}.
This concludes the proposition.
\end{proof}
 
\begin{remark}\label{rem:Q}
If $k$ has characteristic zero, then for any prime $p\geq 3$, Proposition \ref{prop:cohononvanishing:Constr3} holds for the integral polynomial $g=p^2\cdot \left( \sum_{i=0}^nx_i^{\lceil \frac{n+1}{2}\rceil} \right) ^2 + x_0^\epsilon x_0x_1\cdots x_n$.
The proof is essentially the same as the one presented above, where we replace the  characteristic zero degeneration $t\to 0$, by a degeneration to characteristic $p$. 
\end{remark}

\begin{remark}
In the above proof, we specialized $b$ to $\frac{-c_1}{x_0^{d_1}}$.  
Since $q$ becomes isotropic under our specialization, the class $f^\ast \alpha$ from Proposition \ref{prop:unramified-coho} does not stay unramified (because $H^n_{nr}(K(Q)/K,\Z/2)=0$ if $Q$ is rational over $K$).
This is coherent with the observation that in our specialization, the condition (\ref{eq:cond:codim}) that has been used in an essential way in the proof of Proposition \ref{prop:unramified-coho} is heavily violated, because we specialized $g$ to  $-x_0^\epsilon x_0\cdots x_n$.
\end{remark}

\section{A hypersurface singular along an $r$-plane} \label{sec:hypersurface}

In this section we work over an algebraically closed field $k$ of characteristic different from $2$ and put $K=k(\CP^n)$.
Let $n\geq 1$ and $1\leq r\leq 2^n-2$, and let $e_0,\dots ,e_{r+1}\in k[x_0,\dots ,x_n]$ be homogeneous polynomials of degrees $\deg(e_0)=d$ and $\deg(e_i)=d-2$ for $i\geq 1$.
Assume further that the following three conditions are satisfied, where $q$ and $\psi$ are as in (\ref{def:q}) and (\ref{def:psi}), and $g$ satisfies (\ref{eq:cond:codim}) and (\ref{eq:cond:square}).
Firstly, 
\begin{align} \label{con0:coprime}
\text{$\operatorname{gcd}(e_0,\dots ,e_{r+1})=1$, i.e.\ $e_0,\dots ,e_{r+1}$ are coprime in $k[x_0,\dots ,x_n]$.}
\end{align}
Secondly, there is some $\mu\in K^*$ with
\begin{align} \label{con1:qsimei}
q \cong \langle\mu\rangle \otimes \left\langle \frac{e_0}{x_0^{d}},\frac{e_1}{x_0^{d-2}},\frac{e_2}{x_0^{d-2}},\dots ,\frac{e_{r+1}}{x_0^{d-2}}\right\rangle   .
\end{align}
Thirdly, there is some $\lambda\in K^\ast$ with
\begin{align} \label{con2:subpsi}
\langle\lambda\rangle \otimes \left \langle \frac{e_1}{x_0^{d-2}},\frac{e_2}{x_0^{d-2}},\dots ,\frac{e_{r+1}}{x_0^{d-2}} \right \rangle \subset \psi .
\end{align}

Let $N:=n+r$ and choose homogeneous coordinates $x_0,\dots ,x_n,y_1,\dots ,y_{r+1}$ on $\CP^{N+1}$.
We consider the hypersurface
\begin{align} \label{def:Z}
Z=\{F=0\}\subset \CP^{N+1}\ \ \text{with}\ \ F=e_0+\sum_{i=1}^{r+1} e_i \cdot y_i^2 
\end{align}
of degree $d$.
Condition (\ref{con0:coprime}) implies that $Z$ is integral.

Let $\widetilde \CP^{N+1}$ be the blow-up of $\CP^{N+1}$ along the $r$-plane $P:=\{x_0=\dots = x_n=0\}$.
Then, $\widetilde \CP^{N+1}\cong \CP(\mathcal E)$, where $\mathcal E:=\mathcal O_{\CP^n}(-1)\oplus \mathcal O_{\CP^n}^{\oplus (r+1)} $ and the natural morphism $\widetilde \CP^{N+1}\to \CP^n$, induced by projection to the $x$-coordinates,
 identifies to the projection $\CP(\mathcal E)\to \CP^n$.
The blow-up 
$$
Y:=Bl_PZ 
$$ 
of $Z$ along $P$ is the proper transform of $Z$ in $\widetilde \CP^{N+1}$ and so we get a morphism $f:Y\to \CP^n$.
Locally over $\CP^n$,  
$Y\subset \CP(\mathcal E)$ is given by the quadratic form 
\begin{align} \label{def:Y}
e_0 \cdot z_0^2+\sum_{i=1}^{r+1} e_i\cdot z_i^2 ,  
\end{align}
where $z_0$ is a local coordinate that trivializes $\mathcal O_{\CP^n}(-1)$ and  $z_1,\dots ,z_{r+1}$ trivialize $\mathcal O_{\CP^n}^{\oplus (r+1)}$, 
c.f.\ \cite[Section 3.5]{Sch1}.

By condition (\ref{con1:qsimei}), the generic fibre of $f: Y\to \CP^n$ is birational to the quadric over $K$ defined by $q$ in (\ref{def:q}).
Hence, $ f^\ast \alpha\in H^n_{nr}(k(Y)/k,\Z/2)$ is unramified by Proposition \ref{prop:unramified-coho}. 

\begin{proposition} \label{prop:hypersurface}
Let $k$ be an algebraically closed field of characteristic different from $2$.
Let $n\geq 1$ and $1\leq r\leq 2^n-2$.
Let $e_0,\dots,e_{r+1}\in k[x_0,\dots ,x_n]$ be homogeneous polynomials as above such that (\ref{con0:coprime}), (\ref{con1:qsimei}) and (\ref{con2:subpsi}) hold and consider the corresponding hypersurface $Z\subset \CP^{N+1}$ from (\ref{def:Z}).
Let $\tau: Y' \to Y=Bl_PZ$ be an alteration and let $\xi:Y'\to Z$ be the natural morphism.  

Then any subvariety $E\subset \xi^{-1}(Z^{\sing})$ satisfies $(\tau^\ast f^\ast \alpha)|_{E}=0\in H^n(k(E),\Z/2)$. 
\end{proposition}

\begin{proof} 
Consider $Y=Bl_PZ$. 
In the coordinates (\ref{def:Y}), the exceptional divisor $E'\subset Y$ of the blow-up $Y\to Z$ is given by $z_0=0$.
Note that $E'$ might be reducible (e.g.\ if $e_1,\dots ,e_{r+1}$ are not coprime).
However, the generic fibre of $f|_{E'}:E'\to \CP^n$ is a smooth quadric and so $E'$ has a unique component  $E''\subset E'$ that dominates $\CP^n$.
Moreover, condition (\ref{con2:subpsi}) implies that the generic fibre $E''_\eta$ of $E''\to \CP^n$ is defined by a quadratic form that is similar to a subform of $\psi$.
Hence, Theorem \ref{thm:EL} shows 
$$
(f^\ast \alpha)|_{E''}=0\in H^n(k(E''),\Z/2) .
$$

Since $\psi$ is anisotropic, the generic fibre  $E''_\eta$ is a smooth quadric over $K=k(\CP^n)$.
Let $E\subset E''$ be a subvariety which dominates $\CP^n$.
Since $(f^\ast \alpha)|_{E''}=0$, the injectivity theorem (see e.g.\ \cite[Theorem 3.8.1]{CT}), 
applied to the local ring of $E''_{\eta}$ at the generic point of $E$, then shows 
\begin{align} \label{eq:f*alpha=0}
(f^\ast \alpha)|_{E}=0\in H^n(k(E),\Z/2) .
\end{align} 

Let now $\tau:Y'\to Y$ be an alteration.  
The composition $\xi: Y'\to Z$ yields an alteration of $Z$.
Let $E\subset \xi^{-1}(Z^{\sing})$ be a subvariety. 
If $E$ does not dominate $\CP^n$ via $f\circ \tau$, then   $(\tau^\ast f^\ast \alpha)|_{E}=0$ follows from Proposition \ref{prop:unramified-coho}, because condition (\ref{con1:qsimei}) implies that the generic fibre of $f:Y\to \CP^n$ is isomorphic to the projective quadric over $K$, defined by $q$ from (\ref{def:q}).
On the other hand, if $E$ dominates $\CP^n$ via $f\circ \tau$, then $\tau(E)$ must be a subvariety of $E''$, because $\xi(E)\subset Z^{\sing}$.
Since the generic fibre of $Y\to \CP^n$ is smooth, $f^\ast \alpha$ can be restricted to the generic point of $\tau(E)$ (see \cite[Theorem 4.1.1]{CT}) and this restriction vanishes by (\ref{eq:f*alpha=0}) above.
Hence, $(\tau^\ast f^\ast \alpha)|_{E}=0$  by functoriality. 
This finishes the proof of Proposition \ref{prop:hypersurface}.
\end{proof} 

\begin{remark} \label{rem:CH0} 
The exceptional divisor $E'$ of $Y=Bl_PZ$ is a hypersurface of bidegree $(d-2,2)$ in $\CP^n\times \CP^r$, given by $\sum_{i=1}^{r+1}e_i\cdot y_i^2=0$. 
The generic fibre $E'_\eta$ of $E'\to \CP^r$ is thus a hypersurface of degree $\leq d-2$ (with equality if $E'$ is irreducible) in $\CP^{n}_{k(\CP^r)}$ and in general it seems very unlikely that such a hypersurface admits a zero-cycle of degree one.
In particular, it seems unlikely that $Z$ admits a universally $\CH_0$-trivial resolution, cf.\ \cite{CT-Pirutka}.
Hence, the degeneration method in \cite{voisin,CT-Pirutka} does not seem to apply to the singular hypersurface $Z$ considered above. 
We expect in particular that $E'_\eta$ is not stably rational and so also the methods from \cite{NS,KT} (which assume characteristic $0$) do not seem to apply. 
\end{remark}


\section{Proof of main results}

\subsection{Theorem \ref{thm:main:intro}}
Via Lemma \ref{lem:dec}, Theorem \ref{thm:main:intro} follows from the following more general result.

\begin{theorem}\label{thm:main}
Let $k$ be an uncountable field of characteristic different from two.
Let $N\geq 3$ be a positive integer and write $N=n+r$, where $n,r\geq 1$ are integers with $2^{n-1}-2\leq r\leq 2^n-2$. 
Let $L\subset \CP^{N+1}_k$ be either empty or a linear subspace with $\dim(L)\leq r$.
Fix integers $d\geq n+2$ and $d'\leq d-2$.

Then a very general hypersurface $X\subset \CP^{N+1}_k$ of degree $d$ and with multiplicity $d'$ along $L$ does not admit an integral decomposition of the diagonal over the algebraic closure $\overline k$.
\end{theorem}

\begin{proof} 
We aim to reduce to the case where $k$ is algebraically closed.
To this end, consider the parameter space $\mathcal M$ that parametrizes all hypersurfaces in $\CP^{N+1}_k$ of degree $d$ and with multiplicity $d'$ along $L$. 
Suppose we know the theorem over the algebraic closure $\overline k$ of $k$.
Then there are countably many proper subvarieties $S_i\subsetneq \mathcal M_{\overline k}$, defined over $\overline k$, such that for any hypersurface $X$ in $\mathcal M_{\overline k}\setminus \bigcup_i S_i$, $X_{\overline k}$ does not admit an integral decomposition of the diagonal.  
Each $S_i$ is defined over some finite extension of $k$ and so it has finitely many orbits under the Galois group $\Gal(\overline k/k)$.
Hence, up to adding all the Galois conjugates, we may assume that $\bigcup_i S_i$ is closed under the Galois action.
This implies that $\bigcup_i S_i =\bigcup_j (T_j)_{\overline k}$ for a countable union of proper subvarieties $T_j\subset \mathcal M$, that are defined over $k$.
It then follows that any hypersurface parametrized by $\mathcal M \setminus \bigcup_j T_j$ does not admit an integral decomposition of the diagonal over $\overline k$, as we want.
We may thus from now on assume that $k$ is algebraically closed. 

We choose coordinates  $x_0,\dots ,x_n,y_1,\dots ,y_{r+1}$ on $\CP^{N+1}$, such that 
$$
L\subset P:= \{x_0=\dots =x_n=0\} .
$$
In order to prove the theorem, we then specify the hypersurface $Z$ from Section \ref{sec:hypersurface} as follows.
Let $g$ and $b$ be as in (\ref{eq:g:t}) in Section \ref{sec:non-vanishing} and  consider the homogeneous polynomials $c_i$ of degree $d_i$, defined in Section \ref{sec:construction}, where we assume that (\ref{eq:c1:Constr3}) holds, i.e.\ $c_1=x_1\cdots x_n$.

\textbf{Case 1.} $d\geq n+2$ with $d$ even.

Let $h=x_0+x_1$ and consider
\begin{align} \label{eq:ei:Case1}
e_0:=h^{d-\deg(g)} g \ \ \text{and}\ \ e_i:=x_0^{d-2-d_i}c_i 
\end{align}
for $i\geq 1$.
Since $\deg(g)=2\lceil\frac{ n+1}{2}\rceil$ and $d\geq n+2$, $d-\deg(g)$ is non-negative.
Similarly, $d-2-d_i\geq 0$ because $d_i\leq n$ for all $i$. 

Since $G$ contains the monomial $x_i^{\lceil\frac{ n+1}{2}\rceil}$ nontrivial for all $i$, $g=t^2G^2+x_0^\epsilon x_1\cdots x_n$ from (\ref{eq:g:t}) is not divisible by $x_i$ for any $i$.
Hence, $e_0,\dots ,e_{r+1}$ are coprime and so (\ref{con0:coprime}) holds.
Since $d-\deg(g)$ is even, because $d$ and $\deg(g)$ are even, the conditions (\ref{con1:qsimei}) and (\ref{con2:subpsi}) are also satisfied (with $\lambda=\mu=1$).
We may then consider the degree $d$ hypersurface $Z:=\{F=0\}\subset \CP^{N+1}$ from (\ref{def:Z}), where $F=e_0+\sum_{i=1}^{r+1} e_i\cdot y_i^2  $. 

Let now $X$ be a very general hypersurface of degree $d\geq n+2$  with $d$ even as in the theorem.
Then $X$ degenerates to $Z$, see e.g.\ \cite[\S 2.2]{Sch1}.
Let $Y:=Bl_PZ$ and let $f:Y\to \CP^n$ be the morphism induced by projection to the $x$-coordinates.
Since $g$ in (\ref{eq:g:t}) satisfies (\ref{eq:cond:codim}) and (\ref{eq:cond:square}), Proposition \ref{prop:unramified-coho} shows that $f^\ast \alpha \in H^n_{nr}(k(Y)/k,\Z/2)$, where $\alpha$ is the class from (\ref{eq:alpha}).
By Proposition \ref{prop:cohononvanishing:Constr3}, $f^\ast \alpha \neq 0$.
By de Jong and Gabber, there is an alteration $\tau: Y'\to Y$ of $Y=Bl_PZ$ of odd degree; the natural map $\xi: Y'\to Z$ yields an alteration of odd degree of $Z$.  
By Proposition \ref{prop:hypersurface}, the restriction $(\tau^\ast f^\ast \alpha)|_{E}=0$ vanishes for any subvariety $E\subset Y'$ which maps to the singular locus of $Z$. 
It thus follows from Proposition \ref{prop:degeneration} that $X$ does not admit an integral decomposition of the diagonal, as we want.

\textbf{Case 2.} $d\geq n+2$ with $d$ odd.

For $i>0$, we consider $x_1c_i$ and absorb squares; the formal definition is as follows: 
$$
 c_{\epsilon}':=x_1^{1-\epsilon_1}\cdot \prod_{i=2}^nx_i^{\epsilon_i}\ \ \text{and}\ \  c'_i:= c_{\rho(i)}' .
$$
Let $d_i':=\deg( c_i')$ and note that $d'_i\leq n$ for all $i\geq 1$. 
We then define
\begin{align} \label{eq:ei:Case2}
e_0:= h^{d-\deg(g)-1}x_1 g \ \  \text{and} \ \ e_i:=x_0^{d-2-d_i'} c'_i 
\end{align}
for $i\geq 1$, 
where $h=x_0+x_1$ as in case 1. 

Since $c_1'=x_2\cdots x_n$ is not divisible by $x_1$ and $h$, and $g$ is not divisible by $x_i$ for any $i$, $e_0,\dots ,e_{r+1}$ are coprime and so (\ref{con0:coprime}) holds. 
Moreover, since $d-\deg(g)-1$ is even, because $\deg(g)$ is even and $d$ is odd in case 2, $e_0,\dots ,e_{r+1}$ satisfy the conditions (\ref{con1:qsimei}) and (\ref{con2:subpsi}) with $\lambda=\mu=\frac{x_1}{x_0}$.
We may then consider the degree $d$ hypersurface $ Z:=\{F=0\}\subset \CP^{N+1}$ from (\ref{def:Z}), where $F=e_0+\sum_{i=1}^{r+1} e_i\cdot y_i^2  $. 

Let now $X$ be a very general hypersurface of degree $d\geq n+2$ with $d$ odd as in the theorem.
Then $X$ degenerates to $Z$ (see e.g.\ \cite[\S 2.2]{Sch1}) and we conclude as in case 1 that $X$ does not admit an integral decomposition of the diagonal.
This finishes the proof of the theorem.
\end{proof}

If we put $d'=1$ in Theorem \ref{thm:main}, we obtain the particularly interesting case of hypersurfaces that contain a linear subspace of dimension $r'\leq r$. 
If $r'> \frac{N}{2}$, then any such hypersurface is singular, but for $r'\leq \frac{N}{2}$, a general such hypersurface is smooth, as we will recall below. 
This yields the following. 

\begin{corollary} \label{cor:linearspace}
Let $k$ be an uncountable field of characteristic different from two.
Let $N\geq 3$ be a positive integer and write $N=n+r$, where $n,r\geq 1$ are integers with $2^{n-1}-2\leq r\leq 2^n-2$.
Fix integers $d\geq n+2$ and $l\leq \frac{N}{2}$.

Then a very general hypersurface $X\subset \CP^{N+1}_k$ of degree $d$ and containing a linear space of dimension $l$ is smooth and stably irrational over the algebraic closure $\overline k$.
\end{corollary}
\begin{proof}
Apart from the assertion that $X$ is smooth, the corollary is an immediate consequence of the case $d'=1$ in Theorem \ref{thm:main}.
To prove smoothness, it suffices to find a single example of a smooth hypersurface $X\subset \CP^{N+1}_k$ of degree $d\geq 3$ which contains a linear space of dimension $\lfloor\frac{N}{2}\rfloor$.
If $\operatorname{char}(k)\nmid d$, a smooth example is given by $\sum_{i=0}^{N+1} x_i^d=0$, which contains the linear space
$
\{x_0=\zeta x_1,x_2=\zeta x_3,\dots , x_{N}=\zeta x_{N+1} \}
$ 
if $N$ is even 
and
$
\{x_0=\zeta x_1,x_2=\zeta x_3,\dots , x_{N-1}=\zeta x_{N},x_{N+1}=0 \}
$
if $N$ is odd, where $\zeta\in k$ satisfies $\zeta^{d}=-1$.
If  $\operatorname{char}(k)\mid d$, a smooth example is given by $x_0^d+\sum_{i=0}^{N}x_ix_{i+1}^{d-1}=0$, which contains the linear space
$
\{x_0=x_2=\dots =x_{2\lfloor\frac{N+1}{2}\rfloor}=0\}.
$
This concludes the corollary.
\end{proof}

\begin{proof}[Proof of Corollary \ref{cor:unirational}]
Let $N\in \{6,7,8,9\}$.
By Corollary \ref{cor:linearspace}, a very general quintic $X\subset \CP_\C^{N+1}$ containing a $3$-plane is smooth and stably irrational.
On the other hand, at least in characteristic zero, these examples are unirational by \cite{CMM}.
This proves Corollary \ref{cor:unirational}.
\end{proof}

\subsection{Theorem \ref{thm:slope} and examples over $\mathbb F_p(t)$ and $\Q$} 
Examples of stably irrational smooth quartic threefolds over $\overline \Q$ and of some higher-dimensional hypersurfaces over $\Q$ were previously given in \cite{CT-Pirutka} and \cite{totaro}, respectively.
If $N=2^n+n-2$, then we can also obtain examples defined over small fields like $\Q$ or $\F_p(t)$, as follows.
By Lemma \ref{lem:dec}, our result implies Theorem \ref{thm:slope} stated in the introduction.

\begin{theorem} \label{thm:Q}
Let $n\geq 2$ be an integer and put $N:=2^n+n-2$.
Let $k$ be a field of characteristic different from two.
If $k$ has positive characteristic, assume that it has positive transcendence degree over its prime field.
Then for any degree $d\geq n+2$ there is a smooth hypersurface $X\subset \CP^{N+1}_{k}$ of degree $d$ whose base change $X_{\overline k}$ to the algebraic closure of $k$ does not admit an integral decomposition of the diagonal.
\end{theorem}

\begin{proof}
The proof of the theorem follows the same line of argument as the proof of Theorem \ref{thm:main}.
The main difference being that we will degenerate to a hypersurface which is defined over the algebraic closure of a finite field and so the non-vanishing result from Proposition \ref{prop:cohononvanishing:Constr3} does not apply.
We will replace that non-vanishing result by Lemma \ref{lem:bnotsquare}, which requires the assumption $r=2^n-2$ (and so $N=2^n+n-2$).
We explain the details in what follows. 

Let $n\geq 2$, $r=2^n-2$ and $N=n+r$.
Let further $p$ be an odd prime.
We will now work over the algebraic closure $\overline \F_p$ of $\F_p$.
As in the proof of Theorem \ref{thm:main}, we choose coordinates  $x_0,\dots ,x_n,y_1,\dots ,y_{r+1}$ on $\CP^{N+1}$, such that 
$$
L\subset P:= \{x_0=\dots =x_n=0\} .
$$
We consider $G=\sum_ix_i^{\lceil \frac{n+1}{2}\rceil}$, $g=G^2+x_0^\epsilon x_0\cdots x_n$ and $b=\frac{g}{x_0^{\deg(g)}}$.
With these choices, and for any $d\geq n+2$, we then consider the polynomials $e_i$ used in the proof of Theorem \ref{thm:main} (and which depend on the parity of $d$).
These choices determine the hypersurface $Z\subset \CP^{N+1}_{\overline \F_p}$ from Section \ref{sec:hypersurface}.
By definition, $Z=Z_0\times \overline \F_p$ is the base change of a hypersurface $Z_0$ which is defined over the prime field $\F_p$.
Let $Y=Bl_LZ$ with natural map $f:Y\to \CP^n$.
The class $f^\ast \alpha \in H^n(\overline \F_p(Y),\Z/2)$ is unramified over $\overline \F_p$ by Proposition \ref{prop:unramified-coho}.
Moreover, since $b$ is not a square and since $r=2^n-2$, Lemma \ref{lem:bnotsquare} implies that $f^\ast \alpha$ is nontrivial.
By Proposition \ref{prop:hypersurface}, the assumptions of the degeneration method (Proposition \ref{prop:degeneration}) are satisfied by $Z$ and so any proper variety which specializes to $Z$ is not stably rational.
This implies the theorem as follows.

If $k$ has positive characteristic $p$, then by the assumptions in the theorem, it has positive transcendence degree over its prime field.
The generic fibre of a sufficiently general pencil of degree $d$ hypersurfaces over $\F_p$ which contains $Z_0$ gives an example of a smooth hypersurface $X$ of degree $d$ which is defined over $k$ and such that $X_{\overline k}$ degenerates to $Z$.
Hence, $X_{\overline k}$ does not admit a decomposition of the diagonal, as we want.

If $k$ has characteristic zero, then we may choose any prime $p>2$ and consider the hypersurface $Z$  over $\overline \F_p$ from above.
There is a smooth hypersurface $X$ over $k$ (in fact over $\Q$) and of degree $d$ such that $X_{\overline k}$ degenerates to $Z$.
This shows that $X_{\overline k}$ does not admit a decomposition of the diagonal, as we want. 
This concludes the proof.  
\end{proof}

\subsection{Theorem \ref{thm:unramified:coho} and the integral Hodge conjecture for unirational varieties} \label{subsec:IHC:proof}

\begin{proof}[Proof of Theorem \ref{thm:unramified:coho}]
Let $N\geq 3$ be an integer and put $n:=N-1$ and $r=1$.
Consider the polynomial $g\in \C[x_0,\dots ,x_n]$ and the corresponding rational function $b$ from (\ref{eq:g:t}).
We then consider the quadratic form $q=\left\langle b,\frac{c_1}{x_0^{d_1}},\frac{c_2}{x_0^{d_2}}\right\rangle $ from (\ref{def:q}), where we assume that the bijection $\rho$ is chosen in such a way that $c_1=x_1x_2\cdots x_n$ and $c_2=x_2x_3\cdots x_n$. 

Let $Q$ be the projective conic over $K=\C(\CP^n)$ that is defined by $q$.
By Hironaka's theorem, we can choose some smooth complex projective variety $X$ of dimension $N=n+1$ together with a morphism $X\to \CP^n$ whose generic fibre is isomorphic to $Q$.
Our choice of $c_1$ and $c_2$ implies that $q$ is similar to
$
\left\langle b\cdot \frac{x_2x_3\cdots x_n}{x_0^{n-1}}, \frac{x_1}{x_0},1 \right\rangle 
$
and so $X$ is unirational, see e.g.\ \cite[Lemma 14]{Sch1}.
On the other hand, $H^n_{nr}(\C(X)\slash \C,\Z/2)\neq 0$ by Propositions \ref{prop:unramified-coho} and \ref{prop:cohononvanishing:Constr3}.
This proves the theorem in the case where $i=N-1$.
The general case follows by taking products with projective spaces, because unramified cohomology is a stable birational invariant, see \cite{CTO}.
This concludes the proof of Theorem \ref{thm:unramified:coho}.
\end{proof}

\begin{proof}[Proof of Corollary \ref{cor:IHC}]
Corollary \ref{cor:IHC} is a direct consequence of \cite[Th\'eor\`eme 1.1]{CTV} and Theorem \ref{thm:unramified:coho}, which produces unirational smooth complex projective varieties in any dimension at least four with nontrivial third unramified $\Z/2$-cohomology. 
\end{proof}

\section{Supplements}

\subsection{Double covers}
 
It is possible to adapt the arguments of this paper to the case of double covers of projective spaces.
The result is as follows; for earlier results on the rationality problem for double covers, see e.g.\ \cite{kollar2,voisin,bea,CTP2,oka,HPT2}.

\begin{theorem}
Let $N\geq 3$ be a positive integer and write $N=n+r$ with $2^{n-1}-2\leq r\leq 2^n-2$.
Let $k$ be an uncountable field of characteristic different from two.  
Then a double cover of $\CP^N_k$, branched along a very general hypersurface of even degree $d\geq 2\lceil\frac{n+1}{2}\rceil+2$ is not stably rational over the algebraic closure of $k$. 
\end{theorem}

\begin{proof}
As in Theorem \ref{thm:main}, it suffices to treat the case where $k$ is algebraically closed.
Let $x_0,\dots ,x_n,y_1,\dots ,y_r$ be coordinates on $\CP^N$ and consider the $(r-1)$-plane $P=\{x_0=\dots =x_n=0\}$.
Let $e_0,\dots ,e_r\in k[x_0,\dots ,x_n]$ be homogeneous polynomials of degrees $\deg(e_0)=d$ and $\deg(e_i)=d-2$ for all $i\geq 1$.
We then consider the hypersurface of degree $d$ in $\CP^N$, given by
$$
F=e_0+\sum_{i=1}^r e_iy_i^2 .
$$
From now on we assume that $d$ is even and we consider the double covering $Z\to \CP^N$, branched along $\{F=0\}$.
Introducing an additional variable $y_{r+1}$, $Z$ is given by the equation $y_{r+1}^2+F=0$.
Since $F$ vanishes on the plane $P$, $Z$ contains a copy of $P$ and we consider the blow-up $Y:=Bl_PZ$.
It is well-known (see e.g.\ \cite[Section 3.5]{Sch1}) that $Y$ carries the structure of a weak $r$-fold quadric bundle $f:Y\to \CP^n$, which locally over $\CP^n$ is given by the equation
\begin{align} \label{eq:doublecover}
z_{r+1}^2+e_0z_0^2+\sum_{i=1}^r e_iz_i^2=0 .
\end{align}
The exceptional divisor of the blow-up $Y\to Z$ is given by $z_0=0$ and so it is the weak $(r-1)$-fold quadric bundle given by $z_{r+1}^2+\sum_{i=1}^r e_iz_i^2=0$.

In order to adapt the arguments used for hypersurfaces, we need to ensure that the following two conditions hold.
Firstly, there is some $\mu\in K^*$ with
\begin{align} \label{con1:qsimei:double}
q \cong \langle\mu\rangle \otimes \left\langle \frac{e_0}{x_0^{d}},\frac{e_1}{x_0^{d-2}},\frac{e_2}{x_0^{d-2}},\dots ,\frac{e_{r}}{x_0^{d-2}},1\right\rangle   .
\end{align}
Secondly, there is some $\lambda\in K^\ast$ with
\begin{align} \label{con2:subpsi:double}
\langle\lambda\rangle \otimes \left \langle \frac{e_1}{x_0^{d-2}},\frac{e_2}{x_0^{d-2}},\dots ,\frac{e_{r}}{x_0^{d-2}},1 \right \rangle \subset \psi .
\end{align}
The first condition ensures by Propositions \ref{prop:unramified-coho} and \ref{prop:cohononvanishing:Constr3} that $f^\ast \alpha\in H^n_{nr}(k(Y)\slash k,\Z/2)$ is unramified and nontrivial.
Moreover, by the argument in Section \ref{sec:hypersurface}, condition (\ref{con2:subpsi:double}) ensures together with Proposition \ref{prop:unramified-coho} that for any alteration $\tau:Y'\to Y$ of $Y=Bl_PZ$, the class $\tau^\ast f^\ast \alpha$ restricts to zero on the generic point of any subvariety $E\subset Y'$ that maps to the singular locus of $Z$. 

Let $c_i$ be as in Section \ref{sec:construction} and assume that $\rho$ in (\ref{eq:rho}) is chosen such that $c_1=x_1\cdots x_n$ and $c_{r+1}=x_1x_2$.
We consider $x_1x_2c_i$ and absorb squares to obtain polynomials $c_i''$ with $c_{r+1}''=1$.
For an even integer $d\geq \deg(g)+2=2\lceil \frac{n+1}{2}\rceil +2$, we then put
$$
e_0=h^{d-\deg(g)-2}x_1x_2 g\ \ \text{and}\ \ e_i=x_0^{d-\deg(c_i'')-2}c_i'' 
$$
for $i\geq 1$, where $g$ is as in (\ref{eq:g:t}) and $h$ is a linear homogeneous polynomial which is not a multiple of $x_i$ for $i=0,\dots ,n$.
Then conditions (\ref{con1:qsimei:double}) and (\ref{con2:subpsi:double}) are both satisfied (with $\lambda=\mu=\frac{x_1x_2}{x_0^2}$). 
Applying the same argument as in the proof of Theorem \ref{thm:main} shows then that a double cover of $\CP^N$, branched along a very general  hypersurface of degree $d$ does not admit a decomposition of the diagonal and so it is not stably rational. 
This concludes the theorem.
\end{proof}

\subsection{A general vanishing result}
Starting with the work of Artin--Mumford and Colliot-Th\'el\`ene--Ojanguren, many important examples of rationally connected varieties with unramified cohomology are constructed as follows.
One starts with a proper morphism $f:Y\to \CP^n$ whose generic fibre is a smooth quadric $Q$ over $k(\CP^n)$ and chooses $Q$ in such a special way that there is a class $\alpha\in H^n(k(\CP^n),\Z/2)$ whose pullback $f^\ast \alpha\in H^n(k(Y),\Z/2)$ is nontrivial and unramified over $k$, see e.g.\ \cite{artin-mumford,CTO,HPT,Sch1}.

We prove the following vanishing theorem, which shows that in the above situation, the vanishing condition that is needed in the degeneration method in Proposition \ref{prop:degeneration} for varieties which specialize to $Y$ is automatically satisfied.
Our result generalizes  \cite[Proposition 7]{Sch2} and shows that in fact item (\ref{item:prop:unramified-coho:1}) implies (\ref{item:prop:unramified-coho:2}) in Proposition \ref{prop:unramified-coho}.

\begin{theorem} \label{thm:vanishing}
Let $f:Y\to S$ be a surjective morphism of proper varieties over an algebraically closed field $k$ with $\operatorname{char}(k)\neq 2$ whose generic fibre is birational to a smooth quadric over $k(S)$.
Let $n=\dim(S)$ and assume that there is a class $\alpha\in H^n(k(S),\Z/2)$ with $f^\ast \alpha\in H^n_{nr}(k(Y)/k,\Z/2)$. 

Then for any dominant generically finite morphism $\tau:Y'\to Y$ of varieties and for any subvariety $E\subset Y'$ which meets the smooth locus of $Y'$ and which does not dominate $S$ via $f\circ \tau$, we have   $(\tau^\ast f^\ast \alpha)|_E=0\in H^n(k(E),\Z/2)$.
\end{theorem}

In the proof of Theorem \ref{thm:vanishing}, we use the following two results. 

\begin{proposition}[Proposition 8.1 \cite{CT2}]\label{prop:CT}
Let $A\hookrightarrow B$ be a local homomorphism of discrete valuation rings with residue fields $\kappa_A$ and $\kappa_B$ and fraction fields $K:=\Frac(A)$ and $L:=\Frac(B)$. 
Let $f:\Spec B\to \Spec A$ be the corresponding dominant morphism.  

Assume that there is some $\alpha\in H^n(K,\Z/2)$ with $f^\ast \alpha\in H^n_{\text{\'et}}(\Spec B,\Z/2)\subset H^n(L,\Z/2)$.
If $A\hookrightarrow B$ is unramified, then the restriction of $f^\ast \alpha$ to the closed point of $\Spec B$ lies in the image of $f^\ast:H^n(\kappa_A,\Z/2)\to H^n(\kappa_B,\Z/2)$. 
\end{proposition}

\begin{lemma}\label{lem:model}
Let $k$ be an algebraically closed field of characteristic different from $2$.
Let $S$ be a normal variety over $k$ and let $Q$ be a smooth projective quadric over the function field $k(S)$.
Then for any codimension one point $s\in S^{(1)}$, there is an open neighbourhood $U\subset S$ of $s$ and a smooth variety $X$ over $k$ together with a proper morphism $g:X\to U$ whose generic fibre is isomorphic to $Q$ and such that the special fibre $X_s$ of $X$ over $s\in U$ 
has the following property:
for any component $D_i$ of the reduced special fibre $(X_s)^{\red}$, $D_i$ is smooth over $\kappa(s)$, and, if $X_s$ is non-reduced along $D_i$, then $D_i$ is rational over $\kappa(s)$.
\end{lemma}

\begin{proof}
We will frequently use that a variety over $k$ is smooth if and only if it is regular, because $k$ is algebraically closed.
For instance, since $S$ is normal, it is regular in codimension one and so it is smooth away from a closed subset of codimension two. 
It follows that there is a smooth neighbourhood $U\subset S$ of $s$ such that the closure $D:=\overline{\{s\}} \subset U$ is smooth and cut out by a single regular function $\pi\in \mathcal O_S(U)$. 
After possibly shrinking $U$, we may additionally assume that there are nowhere vanishing regular functions $c_i\in \mathcal O_S(U)^\ast$ on $U$ and an integer $1\leq m\leq r+1$ such that the generic fibre of the $U$-scheme
$$
\mathcal Q:=\left\lbrace  \sum_{i=0}^m c_iz_i^2 + \pi \sum_{i=m+1}^{r+1}c_iz_i^2=0\right\rbrace  \subset \CP_U^{r+1}
$$
is isomorphic to $Q$.
If $m=r+1$, then $\mathcal Q$ is smooth over $U$ (because $\operatorname{char}(k)\neq 2$).
It follows that $\mathcal Q$ is smooth over $k$ and the fibre $\mathcal Q_s$ of $\mathcal Q$ over $s$ is smooth over the residue field $\kappa:=\kappa(s)$, as we want. 
If $m=r$, then $\mathcal Q$ is smooth over $k$ and blowing-up the closure of the singular point $x$ of the fibre $\mathcal Q_s$ of $\mathcal Q$ above $s$ yields a model which is smooth over $k$ and whose fibre above $s$ is of the form $Bl_{x}\mathcal Q_s + 2\cdot \CP^r_\kappa$.
Since $Bl_x \mathcal Q_s$ and $\CP^r_\kappa$ are smooth over $\kappa$, the lemma holds in this case.

If $m<r$, then $\mathcal Q$ has singular locus 
$$
Z:=\mathcal Q^{\sing}=\left\lbrace z_0=\dots=z_m=\pi=\sum_{i=m+1}^{r+1}c_iz_i^2=0\right\rbrace  .
$$ 
Since the $c_i$ are nowhere vanishing on $U$, 
$Z$ is a smooth (but, if $m=r-1$, possibly reducible) quadric bundle over $D$, which is contained in the trivial $\CP^{r-m}$-bundle
$$
P:=\left\lbrace z_0=\dots=z_m=0 \right\rbrace \subset \CP_D^{r+1} .
$$ 
The blow-up $Q':=Bl_Z\mathcal Q$ is smooth over $k$, because its exceptional (Cartier) divisor $E$ is smooth over $k$, as it is given by 
$$
E:=\left\lbrace \sum_{i=0}^m c_i|_D\cdot z_i^2+tw=0 \right\rbrace  \subset \CP_Z(\mathcal O_Z(1)^{\oplus(m+1)}\oplus \varphi^\ast \mathcal N_{D/U}\oplus \mathcal O_Z(2)) ,
$$
where $c_i|_D$ denotes the restriction of $c_i$ to $D$ and 
 $z_0,\dots ,z_m$ are local coordinates that trivialize $\mathcal O_Z(1)^{\oplus(m+1)}$, $t$ trivializes locally the pullback $ \varphi^\ast \mathcal N_{D/U}$ of the normal bundle of $D$ in $U$ via the natural map $\varphi:Z\to D$ and $w$ trivializes locally $\mathcal O_Z(2)$; cf.\ \cite[Th\'eor\`eme 3.3]{CTS}.
Note that the fibre $E_s$ of $E$ above the generic point $s\in D$ is smooth over $\kappa$, because $Z_s$ is a smooth quadric over $\kappa$ and all fibres of $E_s\to Z_s$ are quadrics of full rank.

The fibre of $\mathcal Q'\to U$ above $s$ is reduced and
given by  $E_s + \widetilde {\mathcal Q_s}$, where $\widetilde {\mathcal Q_s}$  is the blow-up of the 
quadric cone $\mathcal Q_s=\{\sum_{i=0}^m \overline c_i z_i^2=0\}\subset \CP_\kappa^{r+1} $ in $Z_s$.
Here, $Z_s$ denotes the fibre of $Z\to D$ above $s$ and $\overline c_i$ denotes the image of $c_i$ in $\kappa=\kappa(s)$. 
The exceptional divisor of $\widetilde {\mathcal Q_s}\to \mathcal Q_s $ is in the above coordinates given by $E_s\cap\{t=0\}$, where $E_s$ denotes the fibre of $E\to D$ above $s$.
The singular locus of $E_s\cap\{t=0\}$ is given by $z_0=\dots =z_m=0$, i.e.\ by the intersection of $E_s$ with the 
proper transform $\widetilde P_s\subset \widetilde {\mathcal Q_s}$ of the plane $
P_s:=\{z_0=\dots=z_m=0\} \subset \CP_\kappa^{r+1}
$.
This shows that the singular locus of $\widetilde {\mathcal Q_s}$ is given by $\widetilde P_s$
and a similar analyses shows that $Bl_{\widetilde P_s} \widetilde {\mathcal Q_s}$ is smooth over $\kappa$.

Let $\widetilde P\subset \mathcal Q'$ be the proper transform of $P$.
Since $Z\subset P$ is a Cartier divisor on $P$, $\widetilde P\cong P\cong \CP^{r-m}_D$. 
Since $D$ is smooth over $k$, so is $\widetilde P$.

Let 
$
\mathcal Q'':=Bl_{\widetilde P} \mathcal Q' 
$.
Since $\mathcal Q'$ and $\widetilde P$ are smooth over $k$, so is $\mathcal Q''$.
Let $E''$ be the exceptional divisor of $\mathcal Q''\to \mathcal Q'$ and let $E''_s$ be the fibre of $E''\to D$ above $s$.  
Since $\mathcal Q'$ is smooth over $k$ and the center of the blow-up is given by $\widetilde P\cong \CP^{r-m}_D$, which is a trivial $\CP^{r-m}$-bundle over $D$, we find that $E''_s$ is a Zariski locally trivial $\CP^m$-bundle over $\CP^{r-m}_\kappa$ and so it is smooth and rational over $\kappa$. 
The fibre $\mathcal Q''_s$ of $\mathcal Q''$ over $s$ is reduced along all components apart from $E''_s$, where the multiplicity is two:  
$
\mathcal Q''_s= 2\cdot E''_s+ Bl_{\widetilde P_s\cap E_s}E_s + Bl_{\widetilde P_s} \widetilde {\mathcal Q_s} 
$. 
Since $E_s$ is smooth over $\kappa$ and $\widetilde P_s\cap E_s\cong Z_s$ is a smooth quadric, $Bl_{\widetilde P_s\cap E_s}E_s$ is smooth over $\kappa$. 
As noted above, 
$
Bl_{\widetilde P_s} \widetilde {\mathcal Q_s}
$
is smooth over $\kappa$ as well.
This shows that $X:=\mathcal Q''$ satisfies the conclusion of the lemma, as we want.
\end{proof}

\begin{proof}[Proof of Theorem \ref{thm:vanishing}]
As in the proof of Proposition \ref{prop:unramified-coho}, one reduces (after replacing $Y'$, $Y$ and $S$ by different birational models) to the case where $Y$ and $S$ are normal and $E$ is a divisor on $Y'$ that maps to divisors on $Y$ and $S$.
By functoriality, it thus suffices to prove that $f^\ast \alpha$ restricts to zero on the function field of a given divisor $E\subset Y$ whose generic point maps to a codimension one point $s\in S^{(1)}$.

By assumptions, the generic fibre of $f:Y\to S$ is birational to a smooth quadric $Q$. 
Applying Lemma \ref{lem:model} to the codimension one point $s\in S^{(1)}$, we get an open neighbourhood $U\subset S$ of $s$, a smooth $k$-variety $X$ and a proper morphism $g:X\to U$ whose generic fibre is isomorphic to $Q$. 
Moreover, for any component $D_i$ of the reduced special fibre $(X_s)^{\red}$, $D_i$ is smooth over $\kappa:=\kappa(s)$, and, if $X_s$ is non-reduced along $D_i$, then $D_i$ is rational over $\kappa$. 

We fix some component $D_i$ of $(X_s)^{\red}$ and denote by $x$ its generic point.
We may think about $x$ as a codimension one point on $X$: $x\in X^{(1)}$.
Since the $k$-varieties $X$ and $Y$ are birational (over $S$), the class $g^\ast \alpha \in H^n(k(X),\Z/2)$ is unramified over $k$ by assumptions and so we can restrict $g^\ast \alpha$ to the generic point $x$ of $D_i$.
As before, by slight abuse of notation, we denote this restriction by 
\begin{align} \label{eq:g*alpha|Di}
(g^\ast \alpha)|_{D_i} \in H^n(\kappa(D_i),\Z/2) .
\end{align}
Since $k$ is algebraically closed, $\kappa=\kappa(s)$ has cohomological dimension less than $n$.
We claim that this implies that the above restriction vanishes. 

To prove this claim, let us first deal with the case where $X_s$ is reduced along $D_i$.
We consider the discrete valuation rings $B:=\mathcal O_{X,x}$ and $A:=\mathcal O_{S,s}$.
The morphism $g:X\to S$ induces a local homomorphism $A\hookrightarrow B$, which is unramified because $X_s$ is reduced along $D_i$.
Since $g^\ast \alpha \in H^n(k(X),\Z/2)$ is unramified over $k$, $g^\ast \alpha \in H^n_{\text{\'et}}(\Spec B,\Z/2)$.
Therefore, Proposition \ref{prop:CT} shows that the restriction $(g^\ast \alpha) |_{D_i}$ lies in the image of $g^\ast: H^n(\kappa(s),\Z/2)\to H^n(\kappa(x),\Z/2)$ and so it must vanish because $H^n(\kappa(s),\Z/2)=0$.

Next, we deal with the case where $X_s$ is not reduced along $D_i$.
In this case, $D_i$ is rational over $\kappa$ and so
$$
H^n_{nr}(\kappa(D_i)/\kappa,\Z/2)=H^n(\kappa,\Z/2).
$$
The above right hand side vanishes because $\kappa=\kappa(s)$ has cohomological dimension less than $n$.
To conclude, it thus suffices to see that $(g^\ast \alpha)|_{D_i} \in H^n(\kappa(D_i),\Z/2)$ is unramified over $\kappa$.
To see this, note that $X$ is smooth and integral over $k$.
Since $g^\ast \alpha$ is unramified over $k$, the equivalence of (a) and (b) in \cite[Theorem 4.1.1]{CT} shows that for any scheme point $p\in X$, the class $g^\ast \alpha\in H^n(k(X),\/Z/2)$ comes from a class in $H^n_{\text{\'et}}(\Spec \mathcal O_{X,p},\Z/2)$.
Using functoriality of \'etale cohomology, we conclude that for any scheme point $p\in D_i$, the restriction $(g^\ast \alpha)|_{D_i}\in H^n(\kappa(D_i),\Z/2)$ comes from a class in $H^n_{\text{\'et}}(\Spec \mathcal O_{D_i,p} , \Z/2)$.
The equivalence of (b) and (d) in \cite[Theorem 4.1.1]{CT} applied to the smooth and proper variety $D_i$ over $\kappa$ then shows that $(g^\ast \alpha)|_{D_i}$ is unramified over $\kappa$, as we want.
Altogether, we have thus proven that the class in (\ref{eq:g*alpha|Di}) vanishes.

Up to replacing $Y$ by another normal model, we may assume that the base change $Y_U$ admits a proper birational morphism $Y_U\to X$ over $U$.
In order to prove that $f^\ast \alpha$ vanishes at the generic point of $E\subset Y_U$, it thus suffices to prove that $g^\ast \alpha$ restricts to zero on the generic point of any subvariety $W\subset X$ which lies over the codimension one point $s\in U$.
(That restriction is defined by \cite[Theorem 4.1.1(b)]{CT} because $g^\ast \alpha$ is unramified over $k$ and $X$ is smooth over $k$.)
We have proven this already if $W$ coincides with one of the components of $(X_{s})^{\red}$.
The general case follows then from the injectivity property (see e.g.\ \cite[Theorem 3.8.1]{CT}), because the components of $(X_{s})^{\red}$ are smooth over $\kappa(s)$.
This concludes the proof of the theorem. 
\end{proof}

\appendix

\section{Explicit examples} \label{app:Examples}

In this section we give in any dimension $N$, and in all degrees $d$ covered by Theorem \ref{thm:main:intro}, explicit examples of smooth stably irrational hypersurfaces over countable fields, such as $\Q(t)$ or $\F_p(s,t)$. 
If $N$ is of the special form $N=2^n+n-2$, we produce similar examples over smaller fields, such as $\Q$ or $\F_p(s)$. 
This proves in particular the claims made about the explicit equation in
Section \ref{subsec:equations}.

Let $N\geq 3$ be an integer and write $N=n+r$ with integers $n\geq 2$ and $r\geq 1$ with  $2^{n-1}-2\leq r\leq 2^n-2$.
Let further $d\geq n+2$ be an integer, and let $k$ be a field.

For any $\epsilon \in \{0,1\}^n$, we put $|\epsilon|:=\sum_{i=1}^n \epsilon_i$ and define a bijection $\rho:\{0,1\}^n\to \{1,\dots ,2^n\}$ via $\rho(\epsilon):=1+\sum_{i=1}^{n} (1- \epsilon_i)\cdot 2^ {i-1}$. 
Set $S:=\rho^{-1}(\{1,2,\dots ,r+1\})$ and consider the following homogeneous degree $d$  polynomials:
$$
R_{ev}:=-x_0^{d-n}x_1x_2\cdots x_n+\sum_{\epsilon \in S} 
 x_0^{d-2-|\epsilon|}\cdot x_1^{\epsilon_1}\cdots x_n^{\epsilon_n} \cdot  x^2_{n+\rho(\epsilon)} ,
$$
$$
R_{odd}:=-x_0^{d-n+1}x_2\cdots x_n+\sum_{\epsilon \in S} 
 x_0^{d-2-|\epsilon|+(-1)^{1-\epsilon_1}}\cdot x_1^{1-\epsilon_1}x_2^{\epsilon_2}\cdots x_n^{\epsilon_n} \cdot  x^2_{n+\rho(\epsilon)} ,
$$
$$
R':=\begin{cases}
\sum_{i=0}^{N+1}x_i^d\ \ \text{if $\operatorname{char}(k)\nmid d$;}\\
x_0^d+\sum_{i=0}^{N}x_ix_{i+1}^{d-1}\ \ \text{if $\operatorname{char}(k)\mid d$}.
\end{cases}
$$ 

\begin{example} \label{ex:equation}
Let $t,u,v\in k$.
We consider the homogeneous polynomial $F\in k[x_0,\dots ,x_{N+1}]$, given as follows, where $R_{ev}$, $R_{odd}$ and $R'$ are as above.
\begin{itemize}
\item If $d\geq n+2$ is even,  
\begin{align*}
F:=u\cdot \left(t^2\left( \sum_{i=0}^n x_i^{d/2}\right)^2 +R_{ev}\right)
 + v\cdot R' .
\end{align*} 
\item  If $d\geq n+2$ is odd,  
\begin{align*}
F:=u\cdot \left(t^2x_1\left( \sum_{i=0}^n x_i^{\frac{d-1}{2}}\right)^2 +R_{odd}\right) 
 + v\cdot R' .
\end{align*}
\end{itemize}
\end{example}

\begin{theorem} \label{thm:main:explicit}
Let $N\geq 3$, write $N=n+r$ and let $d\geq n+2$ be integers as above.
Let $k$ be a field with $\operatorname{char}(k)\neq 2$ and of transcendence degree at least one if $\operatorname{char}(k)=0$ and two otherwise.
We specialize the elements $t,u,v\in k$, used in Example \ref{ex:equation}, as follows:
\begin{itemize}
\item If $\operatorname{char}(k)=0$, let $t\in k$ be an element that is transcendental over the prime field of $k$ and let $u,v\in \Z$ be different prime numbers such that $v\nmid d$.
\item If $\operatorname{char}(k)=p\geq 3$, let $s,t\in k$ be elements that are algebraically independent over the prime field $\F_p\subset k$, and let $u,v\in \F_p[s]$ be coprime irreducible polynomials. 
\end{itemize}
Then $X:=\{F=0\}\subset \CP^{N+1}_k$, where $F$ is as in Example \ref{ex:equation}, is smooth and stably irrational over $\overline k$.
\end{theorem}
\begin{proof}
Setting $u= 0$, we we see that $X$ degenerates to $\{R'=0\}$, which is a smooth hypersuface by construction.
Hence, $X$ is smooth.
Let $Z$ be the specialization of $X$, given by $v=0$.
As in the proof of Theorem \ref{thm:main}, Propositions \ref{prop:unramified-coho}, \ref{prop:cohononvanishing:Constr3} and \ref{prop:hypersurface} imply that this hypersurface satisfies the assumptions of Proposition \ref{prop:degeneration}.
(This requires to rename the coordinates $y_1,\dots ,y_{r+1}$ used in Section \ref{sec:hypersurface} by $x_{n+1},\dots ,x_{n+r+1}$.) 
Hence, we conclude that $X$ is not stably rational over $\overline k$, as we want.
\end{proof}

\begin{theorem} \label{thm:slope:explicit}
In the above notation, assume that $r=2^n-2$, i.e.\ $N=2^n+n-2$ and let $d\geq n+2$ be an integer.
Let $k$ be a field with $\operatorname{char}(k)\neq 2$ and of transcendence degree at least one if $\operatorname{char}(k)>0$.
We specialize the elements $t,u,v\in k$, used in Example \ref{ex:equation}, as follows:
\begin{itemize}
\item If $\operatorname{char}(k)=0$, set $t:=1$ and let $u,v\in \Z$ be different primes such that $v\nmid d$.
\item If $\operatorname{char}(k)=p\geq 3$,  let $s\in k$ be transcendental over the prime field $\F_p\subset k$, set $t:=1$ and let $u,v\in \F_p[s]$ be coprime irreducible polynomials. 
\end{itemize}
Then $X:=\{F=0\}\subset \CP^{N+1}_k$, where $F$ is as in Example \ref{ex:equation}, is smooth and stably irrational over $\overline k$.
\end{theorem}

\begin{proof}
Setting $u= 0$, we we see that $X$ degenerates to $\{R'=0\}$, which is a smooth hypersuface by construction.
Hence, $X$ is smooth.
Let $Z$ be the specialization of $X$, given by $v=0$.
As in the proof of Theorem \ref{thm:Q}, Lemma \ref{lem:bnotsquare} and Propositions \ref{prop:unramified-coho} and \ref{prop:hypersurface} imply that this hypersurface satisfies the assumptions of Proposition \ref{prop:degeneration}.
(This requires to rename the coordinates $y_1,\dots ,y_{2^n-1}$ used in Section \ref{sec:hypersurface} by $x_{n+1},\dots ,x_{n+2^n-1}$.) 
Hence, we conclude that $X$ is not stably rational over $\overline k$, as we want.
\end{proof}

\section*{Acknowledgements}   
I am grateful to J.-L.\ Colliot-Th\'el\`ene, B.\ Conrad, B.\ Totaro and to the excellent referees, for many useful comments and suggestions. 
I had useful discussions about topics related to this paper with  O.\ Benoist
and L.\ Tasin. 


\end{document}